\documentclass[12pt]{amsart}
\usepackage{amssymb,verbatim,graphicx}

\newtheorem{theorem}{Theorem}[section]

\newtheorem{corollary}[theorem]{Corollary}
\newtheorem{lemma}[theorem]{Lemma}

\newtheorem{example}[theorem]{Example}
\theoremstyle{definition}
\newtheorem{definition}[theorem]{Definition}

\newtheorem{remark}[theorem]{Remark}
\newcommand{\0}{\emptyset}
\newcommand{\ol}{\overline}

\newcommand{\Complex}{\mathbb{C}}

\newcommand{\Int}{\mbox{int}}
\newcommand{\bd}{\mbox{Bd}\,}
\newcommand{\dia}{\mbox{diam}}
\newcommand{\sh}{\rm{Sh}}
\newcommand{\ch}{\mbox{Ch}}
\newcommand{\e}{\varepsilon}
\newcommand{\al}{\alpha}

\newcommand{\hal}{\hat \alpha}
\newcommand{\ph}{\varphi}
\newcommand{\be}{\beta}
\newcommand{\ga}{\gamma}

\newcommand{\hga}{\hat \gamma}

\newcommand{\si}{\sigma}
\newcommand{\ta}{\theta}

\newcommand{\da}{\delta}
\newcommand{\vp}{\varphi}

\newcommand{\nin}{\not\in}
\newcommand{\imp}{\mbox{Imp}}

\newcommand{\hD}{\widehat{D}}
\newcommand{\C}{\mbox{$\mathbb{C}$}}
\newcommand{\Sph}{\C^\infty}
\newcommand{\RC}{\C^\infty}

\newcommand{\tc}{\widetilde{C}}
\newcommand{\D}{\mbox{$\mathbb{D}$}}
\newcommand{\idisk}{\D^\infty}
\newcommand{\sm}{\setminus}
\newcommand{\hX}{\widehat{X}}

\newcommand{\disk}{\mathbb{D}}

\newcommand{\R}{\mathbb{R}}

\newcommand{\ucirc}{\mathbb{S}^1}
\newcommand{\uc}{\mathbb{S}^1}
\newcommand{\degree}{\text{degree}}

\newcommand{\A}{\mathcal{A}}
\newcommand{\Bc}{\mathcal{B}}
\newcommand{\hx}{\mathcal{H}_f}

\newcommand{\iy}{\infty}
\newcommand{\dg}{\mathrm{degree}}
\newcommand{\diam}{\mathrm{diam}}
\newcommand{\win}{\mathrm{win}}
\newcommand{\G}{\mathcal{G}}
\newcommand{\hg}{\mathcal{H}_g}

\newcommand{\he}{\hat{E}}

\newcommand{\HCD}{\mathrm{hypconv}_{\infty}}

\begin{document}

\date{January 30, 2009}
\title[A fixed point theorem for branched covering maps]
{A fixed point theorem for branched covering maps of the plane}

\author{Alexander~Blokh}

\thanks{The first author was partially
supported by NSF grant DMS-0456748}

\author{Lex Oversteegen}

\thanks{The second author was partially  supported
by NSF grant DMS-0405774}

\address[Alexander~Blokh and Lex~Oversteegen]
{Department of Mathematics\\ University of Alabama at Birmingham\\
Birmingham, AL 35294-1170}

\email[Alexander~Blokh]{ablokh@math.uab.edu}
\email[Lex~Oversteegen]{overstee@math.uab.edu} 


\begin{abstract}
It is known that every homeomorphism of the plane has a fixed point in a
non-separating, invariant subcontinuum. Easy examples show that a branched
covering   map of the plane can be periodic point free. In this paper we show
that any branched covering map of the plane of degree with absolute value at
most two, which has an invariant, non-separating continuum $Y$, either has a
fixed point in $Y$, or $Y$ contains a \emph{minimal (by inclusion among
invariant continua), fully invariant, non-separating} subcontinuum $X$. In the
latter case, $f$ has to be of degree $-2$ and $X$ has exactly three fixed prime
ends, one corresponding to an \emph{outchannel} and the other two to
\emph{inchannels}.
\end{abstract}

\maketitle
\section{Introduction}

By $\C$ we denote the plane and by $\RC$ the Riemann sphere. Homeomorphisms of
the plane have been extensively studied. Cartwright and Littlewood
\cite{cartlitt51} have shown that \emph{each orientation preserving homeomorphism of
the plane, which has an invariant non-separating subcontinuum $X$, must have a
fixed point in $X$}. This result was generalized to \emph{all} homeomorphisms by Bell
\cite{bell78}. The existence of fixed points for orientation preserving
homeomorphisms under various conditions was considered in
\cite{brou12a,brow84a,fath87,fran92,guil94}, and of a point of period two for
orientation reversing homeomorphisms in \cite{boni04}.

In this paper we investigate fixed points of \emph{light open} maps of the
plane. By a Theorem of Stoilow \cite{whyb42}, all such maps have finitely many
critical points and are \emph{branched covering} maps of the plane. In
particular if $\mathcal{C}$ denotes the set of critical points of $f$, then for
each $y\in\C\setminus f(\mathcal{C})$, $|f^{-1}(y)|$ is finite and independent
of $y$. We will denote this number by $d(f)$. All such maps are either
\emph{positively} or \emph{negatively oriented} (see definitions below);
holomorphic maps are prototypes of positively oriented maps. If $f$ is
positively oriented then  the \emph{degree} (of the map $f$), denoted  by
$\dg(f)$, equals $+d(f)$ and if $f$ is negatively oriented then $\dg(f)=-d(f)$.
Easy examples, described in Section~\ref{intro}, show that positively and
negatively oriented branched covering maps of the pane can be periodic point
free.

The following is a known open problem in plane topology \cite{ster35}: {\em
``Does a continuous function taking a non-separating plane continuum into
itself always have a fixed point?"}. Bell announced in 1984 (see also Akis
\cite{akis99}) that the Cartwright-Littlewood Theorem can be extended to
holomorphic maps of the plane. This result was extended in
\cite{fokkmayeovertymc07} to all branched covering maps (even to all perfect
compositions of open and monotone  maps) which are positively oriented. Thus,
if $f:\C\to\C$ is positively oriented branched covering map of the plane and
$X\subset \C$ is a non-separating continuum such that $f(X)\subset X$ then $X$
contains a fixed point. The main remaining question concerning branched
covering maps then is that for negatively oriented maps.

Given a continuum $Y$ in the plane, we denote by $T(Y)$, the \emph{topological
hull of } $Y$, the union of $Y$ and all of the bounded components of $\C\sm Y$.
Also, denote by $U_\iy(Y)$ the unbounded component of $\C\sm Y$. Then $T(Y)=\C\sm U_\iy(Y)$ is
a non-separating plane continuum. In this paper we consider a branched covering
map $f$ of the plane of degree $-2$ and prove the following theorem.

\setcounter{section}{5} \setcounter{theorem}{1}

\begin{theorem}\label{main-intro}
Suppose that $f:\C\to\C$ is a branched covering map of degree with absolute
value at most $2$ and let $Y$ be a continuum such that $f(Y)\subset T(Y)$.
Then one of the following  holds.

\begin{enumerate}

\item The map $f$ has a fixed point in $T(Y)$.

\item The continuum $Y$ contains a \emph{fully invariant indecomposable}
continuum $X$ such that $X$ contains no subcontinuum $Z$ with $f(Z)\subset Z$;
moreover, in this case $\degree(f)=-2$.

\end{enumerate}

\end{theorem}

It follows that in case (2) $f$ induces a covering map $G$ of the circle of
\emph{prime ends} of $T(X)$ with $\mathrm{degree}(G)=-2$ and $T(X)$ has exactly
three \emph{fixed} prime ends and for all of them their \emph{principle set} is
equal to $X$. More precisely, let us consider in the uniformization plane the
complement $\idisk$ to the closed unit disk, and choose a Riemann map
$\vp:\disk^\iy\to \Sph\sm T(X)$ such that $\vp(\iy)=\infty$. Then one of the
fixed prime ends, say, $\al$, corresponds to an \emph{outchannel} (i.e., for
sufficiently small crosscuts $C$ whose preimages in the uniformization plane
separate $e^{2\pi\al}\in \uc$ from infinity,  $f(C)$ separates $C$ from
infinity in $\C\sm T(X)$) and the other two prime ends correspond to
\emph{inchannels} (i.e., for sufficiently small crosscuts $C$ separating the
corresponding points on the unit circle from infinity, $C$ separates $f(C)$
from infinity in $\C\sm T(X)$).

Let us outline the main steps of the proof. By known results we may assume that
$\degree(f)=-2$; we may also assume that $f$ has no fixed points in $T(Y)$.
Bell \cite{bell67} (see also \cite{siek68,ilia70}) has shown that then $Y$
contains a subcontinuum $X$ with the following properties: (1) $X$ is minimal
with respect to the property that $f(X)\subset T(X)$, (2) $f(X)=X$ is
indecomposable, and (3) there exists an \emph{external ray} $R$ to $T(X)$ whose
principal set is $X$. Let $c$ be the critical point of $f$ and $\tau: \C\to \C$
be the map such that $\tau(c)=c$ and $\tau(x)$ is the point $y\ne x$ with
$f(y)=f(x)$ (if $x\ne c$). By \cite{bell78} we assume that $\tau(X)\cap X\ne
\0$. By way of contradiction we assume that $X$ is not fully invariant.

The first important step in the proof is made in
Lemma~\ref{emptyint} where we prove that \emph{$X\cap \tau(X)$ is a first
category subset of $X$}. Krasinkiewicz in \cite{kras74} introduced the notions
of internal and external composants and described important properties of these
objects. His tools are instrumental for the results of Section~\ref{baseprel}.
In Section~\ref{advaprel} we construct a modification of the map $f$, which
coincides with $f$ on $T(X)$ and for which the external ray $R$
 has an \emph{invariant tail}, i.e. a part of $R$ from some point
on to $X$ maps over itself, repelling points away from $X$ in the sense of the
order on $R$. In doing so we use a new sufficient condition allowing one to
extend a function from the boundary of a domain over the domain. The proof
of Theorem~\ref{main-intro} is given in Section~\ref{main-section}. There we study how
the ray $R$ approaches $X$ and use the map on $R$ and the fact that
$\tau(X)\cap X$ is a first category set in $X$ in order to come up with a
sequence of segments of $R$ which map one onto the other and converge to a
proper subcontinuum of $X$,
a contradiction.

\setcounter{section}{1} \setcounter{theorem}{0}

\section{Main notions and examples}\label{intro}

All maps considered in this paper are continuous. We begin by giving some
definitions (avoiding the most standard ones). A map $f:X\to Y$ is {\em
monotone}  provided for each continuum or singleton $K\subset Y$, $f^{-1}(K)$
is a continuum or a point.  A map $f:X\to Y$ is {\em light} provided for each point $y\in
Y$, $f^{-1}(y)$ is totally disconnected. A map $f:X\to Y$ is \emph{confluent}
if for each continuum $K\subset Y$ and each component $C$ of $f^{-1}(K)$,
$f(C)=K$. It is well known \cite{whyb42} that all open maps between compacta
are confluent. In the above situation components of $f^{-1}(K)$ are often
called \emph{pullbacks} of $K$.

Every homeomorphism of the plane is either orientation-preserving or
orientation-reversing. In this section we will recall an appropriate extension
of this result, which applies to \emph{open} and \emph{perfect} maps (see
\cite{fokkmayeovertymc07}).

\begin{definition}[Degree of a map]
Let $f:U \to \C$ be a map from a simply connected domain $U$ into the plane.
Let $S$ be a simple closed curve in $U$, and $p\in U\setminus f^{-1}(f(S))$ a
point. Define $f_{p, S}:S\to\ucirc$ by \[ f_{p, S}(x)=\frac{f(x)-f(p)}{|f(x)-f(p)|}.\]
Then $f_{p, S}$ has a well-defined {\em degree}, denoted $\degree(f_{p, S})$. Note that
$\degree(f_{p, S})$ is  the  winding number $\win(f,S,f(p))$ of $f|_S$ about
$f(p)$.
\end{definition}

\begin{definition}
A map $f:U \to \C$  is from a simply connected domain $U\subset \C$ is
\emph{strictly positively-oriented}  ({\em strictly
negatively-oriented}) if for each $p\in T(S)\setminus f^{-1}(f(S))$ we have
$\degree(f_{p,S})>0$ ($\degree(f_{p,S})<0$, respectively).
\end{definition}

\begin{definition}
A map $f:\C\to\C$ is said to be \emph{perfect} if preimages of compacta are
compacta. A perfect map $f:\C\to\C$ is \emph{oriented} provided for each simple
closed curve $S$ we have $f(T(S))\subset T(f(S))$.
\end{definition}

\begin{remark}\label{simplor}
Every strictly positively- or strictly negatively-oriented map is oriented
because if a point $p$ is such that $f(p)\nin T(f(S))$ for a simple closed
curve $S$, then $\degree(f_{p, S})=0$. Also, if $Y$ is a continuum then
$f(T(Y))\subset T(f(Y))$ as follows from the definition of an oriented map and
continuity arguments.
\end{remark}

The following theorem was established in \cite{fokkmayeovertymc07}:

\begin{theorem}\label{orient}
Suppose that $f:\C\to\C$ is a perfect map. Then the following are
equivalent:

\begin{enumerate}

\item\label{pnorient} $f$ is either strictly positively or
strictly negatively oriented.

\item \label{iorient} $f$ is
oriented.

\item\label{conf}  $f$ is confluent.
\end{enumerate}

\end{theorem}

Let us prove a useful lemma related to Theorem~\ref{orient}. A \emph{branched
covering map} of the plane is a map $f$ such that at all points, except for
finitely many \emph{critical points}, the map $f$ is a local homeomorphism, at
each critical point $c$ the map $f$ acts as $z^k$ at $0$ for the appropriate
$k$, and each point which is not the image of a critical point has the same
number of preimages $d$ (then $\dg(f)$ equals $d$ if $f$ is positively oriented
and $-d$ if $f$ is negatively oriented). By a Theorem of Stoilow \cite{whyb42}
an open light map of the plane is a \emph{branched covering map}.

\begin{lemma}\label{non-deg} Suppose that $f:\C\to \C$ is a perfect map such
that for every continuum $K$ and every component $K'$ of $f^{-1}(K)$ the image
$f(K')$ is not a point. Then $f$ is confluent. If in addition $f$ is light,
then it is open (and hence in this case $f$ is a branched covering map).
\end{lemma}

\begin{proof}
Let $f$ be light and show that then it is open. Suppose that $V$ is an open
Jordan disk, $x\in V$, and $f(x)\in \bd f(V)$. Choose a small semi-open arc $I$ in
$C\sm \ol{f(V)}$ with an endpoint of $\ol{I}$ at $f(x)$, and then choose a
component $J$ of $f^{-1}(I)$ containing $x$. By the assumptions of the lemma,
$J$ is not degenerate. Choose a small disk $V'$ so that $x\in V'\subset
\ol{V'}\subset V$. Then the component $J'$ of $J\cap \ol{V'}$ containing $x$ is
not degenerate. Now, the fact that $f$ is light implies that there are points
of $J'$ mapped into $I\sm \{f(x)\}\subset C\sm \ol{f(V)}$, a contradiction. By
a Theorem of Stoilow \cite{whyb42} then $f$ is a branched covering map of the
plane.

Consider the general case. We can use the so-called \emph{monotone-light}
decomposition. Indeed, consider the map $m$ which collapses all components of
sets $f^{-1}(x)$ to points. Then it follows that $f=g\circ m$ where $g$ is a
light map. By the above this implies that $f$ is a composition of a
monotone map and an open light map of the plane. Clearly, this implies that $f$
is confluent.
\end{proof}

A translation by a vector $\mathbf{a}$ (and a translation by a vector
$\mathbf{a}$ followed by a reflection with respect to an axis non-orthogonal to
$\mathbf{a}$), are obvious examples of plane homeomorphisms which are periodic
point free. Clearly, any polynomial of degree strictly bigger than one, acting
on the complex plane, has points of all periods. The following examples show
that this is not true for all positively oriented branched covering maps of the
plane.

\begin{example}
There exists a degree two positively oriented branched covering map of the
plane which is periodic point free.
\end{example}

We will use both polar $(r,\ta)$ and rectangular  $(x,y)$ coordinates. Set
$\ph(r,\ta)=(r, 2\ta)$. We will look for a map $f$ in the form $h\circ \ph$
with $h(x, y)=(x+T(y), y)$ and $T:\R\to \R$ is a continuous positive function
such t hat $T(s)=T(-s)$. Before we define $T$, let us describe the set $A$ of
all points $(r, \ta)$ such that $(r, \ta)$ and $\ph(r, \ta)$ have the same
$y$-coordinates. Then $r\sin(2\ta)=r\sin(\ta)$. Hence,
$\ta\in\{0,\pi/3,\pi,5\pi/3\}$. So, the set $A$ consists of  the $x$-axis and
two radial straight lines coming out of $(0, 0)$ at angles $\ta=\pm \pi/3$.
Given $s\ne 0$, consider the point $P_s$ of intersection between the horizontal
line $L_s$ of points whose $y$-coordinate is $s$ and the set $A$.

The point $P_s$ is the only point on $L_s$ with $\ph$-image also on $L_s$. Then
the distance between the point $P$ and the point $\ph(P_S)$ (and the origin) is
$2|s|/\sqrt{3}$. Set $T(s)=2|s|/\sqrt{3}+2$. Then $f(x, s)\ne (x, s)$ for any
point of $L_s$ because all points of $L_s\setminus \{P_s\}$ map off $L_s$ by
$\ph$, and hence, by the construction, by $f$. On the other hand, $f$
translates $P_s$ two units to the right. Hence $L_s$ does not contain fixed
points. Moreover, since $f(x,0)=(|x|+2,0)$, $f$ also has no fixed points on the
$x$-axis.

To see that $f$ has no periodic points\footnote{We are indebted to M.
Misiurewicz for suggesting this argument.}, let $B$ be the set of points in
$\Complex$ whose argument is in $(-\pi/3,\pi/3)$. Then
$f(\ol{B})=f(\ol{-B})\subset B$ coincides with the shift of the entire set
$\ol{B}$ to the right by two units ($f(\ol{B})=f(\ol{-B})$ because
$\ph(B)=\ph(-B)$). Let $C=\Complex\sm [B\cup -B]$, and let $\imp(z)$ denote the
imaginary part of $z$. If $z\in \Int C$, then $|\imp(f(z))|<|\imp(z)|$ and if
$z\in \bd C$, then $|\imp(f(z))|=|\imp(z)|$. Let us show that a point $z\in
C$ cannot stay in $C$. Indeed, otherwise $y$ has to converge to points of $C$.
However if $y$ were one of these points, then by continuity we would have
$|\imp(f(y))|=|\imp(y)|$ which would imply that $y\in \bd C$ and hence that
$f(y)\in B$ contradicting the assumption that $z$ stays in $C$.

Hence, for every $z\in\Complex$, there exists $n$ such that $f^n(z)\in B$.
Since $f(B)\subset B$, the trajectory stays in $B$ forever. If there exists $m$
such that $f^m(z)$ belongs to the real line, then it converges to $+\infty$. To
study the orbit of a point $z\in B$ which does not belong to the real line,
observe that there exists an increasing function $\xi:\mathbb{R}_+\to
\mathbb{R}_+$ such that if $z\in \ol{f(B)}$ then $|\imp(f(z))|\ge
|\imp(z)|+\xi(|\imp(z)|)$. Therefore if $z\in \ol{f(B)}$ does not belong to the
real line then $|\imp(f^k(z)|\to \infty$ as $k\to\infty$. Hence in fact
$|f^k(z)|\to \infty$ for any point $z$ and $f$ has no periodic points.

The example above can be easily modified to obtain a periodic point free
branched covering map of degree $-2$.

\section{Basic preliminaries}\label{baseprel}

A continuum $X$ is called \emph{indecomposable} if $X$ cannot be written as the
union of two proper subcontinua. Also, $Z$ is \emph{unshielded} if
$Z=\bd (U_\iy(Z))$. We argue by way of contradiction, therefore the
following is our main assumption.

\smallskip

\noindent \textbf{Main Assumption.} \emph{The map $f:\C\to\C$ is a branched
covering map and $Y\subset \C$ is a continuum such that $f(Y)\subset T(Y)$ and $f|_{T(Y)}$ is
fixed point free.}

\smallskip

\noindent Bell \cite{bell67} has shown that in this case $Y$ contains a
subcontinuum $X$ which is minimal with respect to the property that
$f(X)\subset T(X)$ (then, clearly, $f(X)\subset T(X)$) and which must have the
following properties (see \cite{siek68,ilia70} for alternative proofs):

\begin{enumerate}

\item[(A0)] $X$ is minimal among continua $Z\subset Y$ such that $f(Z)\subset T(Z)$;

\item[(A1)] $f(X)=X$ and $T(X)$ is fixed point free;

\item[(A2)] there exists a curve $R_\be$ (a \emph{conformal external ray}, see below) in
    $U_\iy(Z)$ such that $X=\ol{R_\be}\sm R_\be$ (so that $X$ is unshielded and
    has empty interior);

\item[(A3)] $X$ is indecomposable.

\end{enumerate}

We will use $X$ \emph{exclusively} for a continuum with the just listed
properties (A0) - (A3) which are ingredients of the standing assumption on $X$.
Our main aim is to show that then $X$ is fully invariant (i.e., $f^{-1}(X)=X$).
Thus, by way of contradiction we can add the following to our standing
assumption (as we progress, the standing assumption will be augmented by other
ingredients as well).

\begin{enumerate}

\item[(A4)] The set $X$ is not fully invariant.

\end{enumerate}

Below we list well-known facts from  Carath\'eodory theory. Good sources are
the books \cite{miln00} and \cite{pom92}. Let $\disk$ be the open unit disk in
the complex plane and $\disk^\iy=\Sph\sm\ol{\disk}$. Let $\vp:\disk^\iy\to
\Sph\sm T(X)$ be a conformal map such that $\vp(\iy)=\infty$. An \emph{external
ray} $R_\al=\vp(\{re^{2\pi i \al}\mid r>1\})$ is the $\vp$-image of the radial
line segment $r_\al=\{re^{2\pi i \al}\mid r>1\}$. Clearly, an external ray is
diffeomorphic to the positive real axis. If an external ray $R$ compactifies on
a point, say, $x$ (i.e., $\ol{R}\sm R=\{x\}$ then $R$ is said to \emph{land on
$x$}. For convenience we extend the map $\vp$ onto all angles whose rays land:
if the ray $R_\al$ lands at a point $x$, we set $\vp(e^{2\pi i \al})=x$.
Observe, that the extended map $\vp$ is not necessarily continuous at angles
whose rays land. Still, this extension is convenient and will be used in what
follows.

A \emph{crosscut} $C$ (of $T(X)$ or of $U_\iy(X)$) is an open arc in $\C\sm T(X)$ whose closure
is a closed arc with its endpoints in $T(X)$.  If $C$ is a crosscut, then by
the \emph{shadow of} $C$, denoted $\sh(C)$, we mean the bounded component of
$\Sph\sm [T(X)\cup C]$. Sometimes the crosscut which gives rise to a shadow is
said to be the \emph{gate} of the shadow. In the uniformization plane we
consider $\ol{\disk}$ as a marked continuum which allows us to talk about
crosscuts of $\idisk$ too. Moreover, given a crosscut in $\idisk$ we can then
talk about its shadow etc. It is known that if $C$ is a crosscut of $T(X)$,
then $\vp^{-1}(C)$ is a crosscut of $\bd\idisk$, and
$\vp^{-1}(\sh(C))=\sh(\vp^{-1}(C))$.

We say that a crosscut $C$ is an
$R_\al$-\emph{essential crosscut} if $R_\al\cap C$ is a single point, called
the \emph{central point}, and the intersection of $C$ and $R_\al$ is
transverse. A sequence of crosscuts $\{C_i\}$ of $T(X)$ is a \emph{fundamental
chain} provided $C_{i+1}\subset \sh(C_i)$, $\ol{C_{i+1}}\cap \ol{C_i}=\0$ for
each $i$, and $\lim\dia(C_i)=0$. Two fundamental chains $Q=\{q_n\}$ and
$Q'=\{q'_n\}$ are said to be \emph{equivalent} if $\sh(q_n)$ contains all but
finitely many crosscuts $q'_n$, and $\sh(q'_n)$ contains all but finitely many
crosscuts $q_n$. A \emph{prime end} of $U_\iy(X)$ is an equivalence class of
fundamental chains; a fundamental chain is said to \emph{belong} to its prime end.

Given a fundamental chain $\{C_i\}$, the set $\lim \vp^{-1}(C_i)$ is a point
$e^{2\pi i \al}\in \bd \idisk, \al\in[0,1)$; the corresponding prime end then
may be identified with the angle $\al$. Given a prime end $\al$ and a
corresponding fundamental chain $\{C_i\}$, denote by $\imp(\al)$, called the
\emph{impression} of $\al$, the set $\cap \ol{\sh(C_i)}$; it is known that
$\imp(\al)$ does not depend on the choice of a fundamental chain and therefore
is well-defined. Also, consider the set $\Pi(\al)=\ol{R_\al}\sm R_\al$, called
the \emph{principal set} of $R_\al$ (or just of $\al$). It is known that
$\Pi(\al)\subset \imp(\al)$ and that for each point $x\in \Pi(\al)$ there
exists a fundamental chain $C_i$ of the prime end $\al$ such that $C_i\to x$.

The last claim can be improved a little. It was shown in \cite{bo06} that given
an angle $\al$, there exists for each $z\in\R_\al$ an $R_\al$-essential
crosscut $C_z$ such that $\lim\dia\, C_z=0$ as $z\to X$. We call such a family
$C_z$ a \emph{defining family of crosscuts} of the prime end $\al$. For
convenience we order each $R_\al$ so that $x<_\al y$ if and only if the subarc
of $R_\al$ from $y$ to $\infty$ is contained in the subarc of $R_\al$ from $x$
to $\infty$ (thus, as the points move along $R_\al$ from infinity towards $X$,
they decrease in the sense of the order on $R_\al$). Denote by $(a, b)_\al$ the
set of points in $R_\al$ enclosed between the points $a, b\in R_\al$. Also, set
$(0, a)_\al=\{x\in R_\al: x<_\al a\}$. Similarly we define semi-open and closed
subsegments of $R_\al$ when possible (e.g., if $R_\al$ lands, it makes sense to
talk of $[0, a]_\al$ however otherwise the set $[0, a]_\al$ is not defined).
Also, similarly we define relations $\le_\al, >_\al$ and $\ge_\al$. By a
\emph{tail} of $R_\al$ we mean the set of all points $y\in R_\al$ such that
$y<_\al z$ (or $y\le_\al z$) for some $z\in R_\al$.

It is well known that the geometry of the ray $R_\be$ in (A2) and the continuum $X$ is
quite complicated. The ray approaches $X$ so that on either side of $R_\be$ the
distance to $X$ goes to $0$ while it simultaneously accumulates upon the entire
$X$. It then follow from properties of conformal maps that round balls, disjoint
from $X$ but non-disjoint from $R_\be$, with points of the intersection with
$R_\be$ approaching $X$ must go to $0$ in diameter. One can say that $R_\be$
``digs a dense channel'' in the plane eventually accumulating on $X$ by (A2).

As was explained in the Introduction, the main remaining question concerning fixed
point problem for branched covering maps is that dealing with negatively
oriented maps. Since we are interested in maps $f$ such that $|\degree(f)|\le
2$ and thanks to a theorem of Bell \cite{bell78} we may make the following
assumption.

\begin{enumerate}

\item[(A5)] From now on we assume that $f$ is of degree $-2$.

\end{enumerate}

Then $f$ has a unique \emph{critical point}, denoted by $c$, and a
unique \emph{critical value}, denoted by $v=f(c)$. Let $\tau:\C\to\C$ be the
involution defined by $\tau(c)=c$ and if $x\ne c$, $\tau(x)=x'$ where
$\{x'\}=f^{-1}(f(x))\sm \{x\}$. Clearly, $\tau^2=id$ (i.e., the map $\tau$ is
an idempotent homeomorphism of the plane); sometimes we call $\tau(z)$ the
\emph{sibling} of $z$. Let us establish basic properties of $f$ in the
following lemma (some of the properties hold in more general situations,
however we do not need such generality in this paper).

\begin{lemma}\label{pull} The following facts hold.

\begin{enumerate}

\item If $Z$ is a continuum then $f(\Int(T(Z)))\subset \Int(T(f(Z))$.

\item Suppose that $K$ is a non-separating continuum. If $v\nin K$ then there are
exactly two pullbacks of $K$ which are  disjoint and map onto $K$
homeomorphically (on their neighborhoods). If $v\in K$ then $f^{-1}(K)$ is the
unique pullback of $K$ which must contain $c$.

\item If $Y$ is a continuum and $C$ is a pullback of $Y$ then
$T(C)$ is a pullback of $T(Y)$ (and hence $f(T(C))=T(Y)$). In particular, a
pullback of a non-separating continuum is non-separating.


\item Suppose that $U$ is a simply connected domain such that $f(U)$ is
    also simply connected. If $f|_U$ is not a homeomorphism then $U$ must
    contain a critical point.

\end{enumerate}

\end{lemma}

\begin{proof}
(1) Suppose otherwise. Then there is a point $x\in \Int(T(Z))$ such that
$f(x)\nin \Int(T(f(Z)))$. By Remark~\ref{simplor} $f(x)\in T(f(Z))$. Since $f$
is open, we can then find a point $y\in \Int(T(Z))$ such that $f(y)$ is outside
$T(f(Z))$ contradicting Remark~\ref{simplor}.

(2) If $v\nin K$ we can take a curve $Q$ from $v$ to infinity disjoint from
$K$. Then $f^{-1}(Q)$ is a curve which cuts $\C$ into two open half-planes and
is disjoint from $f^{-1}(K)$. Also, each half-plane maps onto $\C\sm K$
homeomorphically. Thus, in this case $f^{-1}(K)$ consists of two components
each of which maps onto $K$ homeomorphically (on sufficiently small
neighborhoods of the pullbacks). Suppose that $v\in K$. Then $f^{-1}(K)$ cannot
have more than one component because $f$ is confluent (hence each pullback of
$K$ maps onto $K$) and $v$ has a unique preimage $c$.

(3) Let us apply (1) to $T(Y)$. If $v\nin T(Y)$ then $C$ must be a homeomorphic
pullback of $Y$ which implies the desired. Let $v\in T(Y)$ and set
$Z=f^{-1}(T(Y))$ where by (1) $Z$ is the unique pullback of $T(Y)$. Let us show
that then $C$ is the unique pullback of $Y$. Let $Y'$ be the boundary of
$U_\iy(T(Y))$, $Z'$ be the boundary of $U_\iy(Z)$. It follows that
$f(Z')\subset Y'\subset Y$. On the other hand, by (1) no point of $\Int(Z)$ can
map to a point of $Y'$, and by the construction no point from $\C\sm Z$ can map
to a point of $Y'$. Hence $Z'=f^{-1}(Y')$.

This implies that $f^{-1}(Y)$ is connected. Indeed, suppose otherwise. Then
there are two pullbacks $Y_1, Y_2$ of $Y$ each of which maps onto $Y$. This
implies that $Z'=(Y_1\cap Z') \cup (Y_2\cap Z')$ which contradicts the fact
that $Z'$ is a continuum. Thus, $C=f^{-1}(Y)$. Moreover, $Z'\subset C$ and
hence $T(Z')=Z\subset T(C)$. On the other hand, by Remark~\ref{simplor},
$f(T(C))\subset T(Y)$ and so $T(C)\subset Z$. Hence $T(C)=Z$. If $Y$ is
non-separating, then $T(Y)=Y$. By the above $T(C)$ is a pullback of $T(Y)=Y$
containing $C$, that is $C$. Thus, $C$ is non-separating as desired.

(4) It immediately follows from (2) that $c\in \ol{U}$ and $v\in f(\ol{U})$.
However we need to show that $c\in U$. Take $x, y\in U$ such that $f(x)=f(y)=z$
and connect them with an arc $I\subset U$. Then $f(I)\subset f(U)=V$. Since $V$
is simply connected by the assumption of the lemma, $T(f(I))\subset V$. Since
$f(x)=f(y)$, it follows
that $J=f^{-1}(f(I))$ is the unique pullback of $f(I)$ (because $f$ is
confluent and both preimages of $z$ belong to $I$). By (2) and (3), $T(J)$ is
the unique pullback of $T(f(I))$, and $v\in T(f(I))$. Since $T(f(I))\subset
V=f(U)$, the unique preimage $c$ of $v$ belongs to $U$ as desired.
\end{proof}

Suppose that $v\nin T(X)$. Then by Lemma~\ref{pull} it follows that there
exists a neighborhood $U$ of $T(X)$ on which $f$ is a homeomorphism. Again, by
the Bell's results \cite{bell78} this implies the existence of an $f$-fixed
point $x\in T(X)$. (Alternatively, the proof given below in Section~\ref{main-section} can
easily be adapted to cover this case.) Therefore we extend our standing
assumption as follows.

\begin{enumerate}

\item[(A6)] From now on we assume that $v\in T(X)$.

\end{enumerate}

Since $f$ is oriented and $f(X)=X$, $f(T(X))\subset T(f(X))=T(X)$. By
Lemma~\ref{pull} the continuum $\hX=f^{-1}(T(X))\supset T(X)$ is non-separating
and maps onto $T(X)$ in a two-to-one fashion (except for the point $c$). Let us
list simple consequences of our standing assumption as applies to $X$ in this
case.

\begin{lemma}\label{simplex} The set $T(X)$ is not fully invariant,
$\tau(X)\not\subset T(X)$, and $X\cap \tau(X)\ne \0$.
\end{lemma}

\begin{proof} Let us show that $T(X)$ is not fully invariant. By
Lemma~\ref{pull} no point from the interior of $T(X)$ can map to $X$ (recall
that $X$ is unshielded by (A2) and hence no point of $X$ belongs to
$\Int(T(X))$. Hence if $T(X)$ is fully invariant then so is $X$ contradicting
(A4). This implies that there are points of $\tau(T(X))$ outside $T(X)$, and
hence $\tau(X)=\bd \tau(T(X))$ cannot be contained in $T(X)$.

Finally, suppose that $X\cap \tau(X)=\0$. Since there are points of $\tau(X)$
outside $T(X)$, this implies that $\tau(T(X))=T(\tau(X))$ is disjoint from
$T(X)$. Hence $f|_{T(X)}$ is a homeomorphism and so $f(T(X))=T(X)$. However
then by (A6) we have $c\in T(X)$, a contradiction with $T(X)$ and $\tau(T(X))$
being disjoint. We conclude that $X\cap \tau(X)\ne \0$ as desired.
\end{proof}

For convenience let us make the conclusions of Lemma~\ref{simplex} a part of
our standing assumption.

\begin{enumerate}

\item[(A7)] $T(X)$ is not fully invariant,
$\tau(X)\not\subset T(X)$, and $X\cap \tau(X)\ne \0$.

\end{enumerate}



A \emph{composant} of $x$ in a continuum $Y$ is the union of all proper
subcontinua of $Y$ which contain $x$. If $Y$ is indecomposable then any two
composants of $Y$ are either equal or disjoint; clearly, if $g(Y)\subset Y$ for
some continuous map $g$, then the image of a composant, being a connected set,
either coincides with $Y$, or is contained in a composant of $Y$. It follows
from the definition that if $Z$ is a composant of $Y$ then for each $p,q\in Y$,
there exists a subcontinuum $P\subset Y$ such that $p,q\in P$. It is well known
\cite{kras74} that, if $Y$ is indecomposable, then each composant in a dense
first category $F_\sigma$ subset of $Y$. By the Baire Category Theorem there are
uncountably many distinct composants in an indecomposable continuum.

Again, assume that $Y$  is indecomposable. A composant $Z$ of $Y$ is
\emph{internal} if every continuum $L\subset\C$ which meets $\C\sm Y$ and $Z$,
intersects \emph{all} composants of $Y$. Equivalently, a composant $Z$ of $Y$
is internal if and only if  every continuum $L\subset\C$ which meets $\C\sm Z$ and
$Z$, intersects all composants of $Y$ (indeed, if $C\subset Y$ meets $Z$ and
$Y\sm Z$ then $C$ must coincide with $Y$). A composant which is not internal is
called \emph{external}.  We denote the union of all external composants by
$E^*_Y$. Clearly, an internal composant $Z$ does not contain accessible points:
if $z\in Z$ is an accessible (from $\C\sm Y$) point, then we can choose an arc
in the appropriate component of $\C\sm Y$ with an endpoint at $z$ which
intersects $Y$ only at $z$, a contradiction with $Z$ being internal. Thus, an
external composant is a generalization of a composant containing an accessible
(from $\C\sm Y$) point. The following results are due to Krasinkiewicz
\cite{kras74}.

\begin{lemma}[Krasinkiewicz]\label{kras}
Let $Y$ be an indecomposable continuum in the plane. Then the following claims hold.

\begin{enumerate}

\item\label{FS} The set $E^*_Y$ is a first category $F_\sigma$-subset of
    $Y$; hence, the union of all internal composants is a dense
    $G_\delta$-set in $Y$.

\item\label{cut}
Let $C$ be an internal composant of $Y$. If $L$ is any continuum which meets
$C$, the complement of $Y$ and does not contain $Y$, then there exists a
neighborhood $U$ of $L$ and a continuum $Z\subset C$ which separates $U$
between two distinct points of $L$.

\end{enumerate}

\end{lemma}

A plane continuum is called \emph{tree-like} if it is one-dimensional (has no
interior) and non-separating. Recall that $X$ is an invariant continuum which
is minimal among continua with respect to the property that $f(X)\subset T(X)$.
The set $X$ has a number of properties listed in the beginning of this section,
in particular it is indecomposable (and hence one-dimensional) and unshielded.
We now study other properties of $X$. First we need a few technical lemmas.

\begin{lemma}\label{fsigde} Suppose that $A\subset X$ is a dense $G_\da$-subset
of $X$. Then $f(A)$ is not a first category subset of $X$.
\end{lemma}

\begin{proof} First let us show that there exists a point $x\in X$ and a small
neighborhood $U$ of $x$ such that $f|_U$ is a homeomorphism onto its image and
$f(X\cap U)=f(U)\cap X$. Indeed, it is obvious if $X$ is fully invariant.
Otherwise choose a point $x\in X$ so that $\tau(x)\nin X$ (i.e., $f(x)$ has a
unique preimage in $X$, namely $x$, and $x$ is not critical). If now $U$ is a
sufficiently small neighborhood of $x$, then $\tau(U)\cap X=\0$; hence
$f(U)\cap X$ consists of points of $X$ which cannot have preimages in
$\tau(U)\cap X=\0$, but must have some preimages because $f(X)=X$. The only
preimages points of $f(U)\cap X$  are those in $U\cap X$ which implies
that $f(X\cap U)=f(U)\cap X$ as desired.

Now, by the conditions of the lemma $A\cap (X\cap \ol{U})$ is a $G_\da$-subset
of $X\cap \ol{U}$, hence by the previous paragraph $f(A\cap (X\cap \ol{U}))$ is
a $G_\da$-subset of $X\cap f(U)$. Therefore $f(A)$ cannot be a first category subset
of $X$.
\end{proof}

The next lemma uses Lemma~\ref{fsigde}.

\begin{lemma}\label{all1} Suppose that $D$ is a union of some composants of $X$
which is a first category subset of $X$. Then there exists a composant $T$ of $X$
such that $f(T)\cap D=\0$.
\end{lemma}

\begin{proof} By way of contradiction suppose that for every composant
$T$ the image $f(T)$ of $T$ intersects $D$. Take a composant $Q$ whose image
contains points not from $D$. By the assumption $Q$ also has points mapped into
$D$. Hence $f(Q)$ contains points of a least two distinct composants which
implies that $f(Q)=X$. If there is another composant $R$ such that $f(R)$ is
not contained in $D$, then it again follows that $f(R)=X$. Since $f$ is
two-to-one, this implies that $X=Q\cup R$ which is impossible because there are
countably many pairwise disjoint composants of $X$. Hence for any composant
$R\ne Q$ we have $f(Q)\subset D$. However, the union of all composants except
for $Q$ is a $G_\da$-subset of $X$ while the set $D$ is a first category subset of $X$.
By Lemma~\ref{fsigde} this is impossible.
\end{proof}

The next lemma studies the images of composants.

\begin{lemma}\label{intint} Suppose that $Z$ is an internal composant of $X$.
Then $f(Z)$ is an internal composant of $X$.
\end{lemma}

\begin{proof}

By way of contradiction suppose that $f(Z)$ is not an internal composant
of $X$. There are two possibilities this can fail: (1) $f(Z)$ is an external
composant of $X$ or $f(Z)=X$, or (2) $f(Z)$ is contained in a composant with
which it does not coincide. We consider these possibilities separately.

(1) Suppose that $f(Z)$ is either an external composant or the entire $X$. Then
there exists a continuum $B$ such that $B\cap f(Z)\ne\0\ne B\sm X$ and $B\cap
X$ is contained in the union of some, but not all, composants of $X$. Then all
composants which intersect $B$ are external, and hence their union $D$ is a
first category subset of $X$.

Choose a point $z\in B\cap f(Z)$ and then a point $z'\in Z$ such that
$f(z')=z$. Then choose the pullback $B'$ of $B$ which contains $z'$. Since $Z$
is an internal composant and $B'$ meets $\C\sm X$, by definition $B'$
intersects all composants of $X$. Hence all composants of $X$ have points
mapped into the union $D$ of some composants which is a first category subset of
$X$. By Lemma~\ref{all1}, this is impossible.

(2) Suppose that $f(Z)\subsetneqq Y$ where $Y$ is a composant of $X$. Choose a
subcontinuum $E\subset Y$ which contains a point $z\in f(Z)$ and a point of
$y\in Y\sm f(Z)$. Choose the pullback $E'$ of $E$ which contains a point $z'\in
Z$ such that $f(z')=z$. Then $E'\not \subset X$ because otherwise it would be
contained in $Z$ and its image would not contain $y$. Hence $E'$ meets points
of $\C\sm X$ and contains the point $z'\in Z$ which implies that $E'$
intersects all composants of $X$ (because $Z$ is an internal composant). Since
$f(E')=E$, then this implies that all composants of $X$ have points mapped into
$Y$, and $Y$, being a composant of $X$, is a first category subset of $X$.
Again, by Lemma~\ref{all1} this is impossible.
\end{proof}

In what follows we will use the following lemma which studies the set
$\tau(X)\cap X$. In the proof we rely upon the above developed tools.

\begin{lemma}\label{emptyint}
The set $\tau(X)\cap X$ is contained in the union $E^*_X$ of all external
composants of $X$. In particular, $\tau(X)\cap X$ is a proper closed subset of
$X$ with empty interior in $X$.
\end{lemma}

\begin{proof} By (A7) $\tau(X)\cap X\ne\0$ and $\tau(X)\sm X\ne\0$.
Choose $x\in \tau(X)\sm X$, then $\tau(x)\in X\sm \tau(X)$ and,
hence, $\tau(X)\cap X$ is a proper, closed subset of  $X$. The fact that
$\tau(X)\cap X$ has empty interior in $X$ is much less trivial.

Let us first show that at most countably many composants of $X$ contain a
subcontinuum which separates $\C$. Indeed, if $C\subset Z$ is a separating
continuum, we can associate to $Z$ a bounded component $V_Z$ of $\C\sm C$.
Since $X$ is unshielded, the  sets $V_Z, V_Q$ for distinct composants $Z, Q$ of
$X$ are disjoint. Hence at most countably many composants of $X$ contain a
subcontinuum which separates $\C$. Also, there is exactly one composant which
contains the critical value $v$. By Lemma~\ref{kras}(\ref{FS}) and since
each composant is a first category $F_\si$-subset of $X$ (see Lemma~2.1 of
\cite{kras74}), the union of the above listed countably many composants and all
external composants of $X$ is still a first category $F_\si$-subset of $X$. Its
complement is the union $I^*_X$ of points of all composants from the collection
$I_X$ of all internal composants of $X$ for which every subcontinuum is
tree-like not containing the critical value $v$. Thus, $I^*_X$ is still a dense
$G_\delta$ subset of $X$.

By Theorem~4.2 of \cite{chmatuty06}, $f(E^*_X)$ is a first category
$F_\sigma$-subset of $X$. Hence we can choose a point $y\in I^*_X\sm f(E^*_X)$. Let
$Y$ be the internal composant of $X$ which contains $y$. Choose $z\in
f^{-1}(y)\cap X$, then $z$ is contained in an internal composant $Z$ of $X$. By
Theorem 5.5 of \cite{roge98}, $f^{-1}(Y)=Y_1\cup Y_2$ such that $Y_1\cap
Y_2=\0$ and for each $i$ and each $p,q\in Y_i$, there exists a subcontinuum
$P\subset Y_i$ such that $p,q\in P$. Moreover, the map $f|_{Y_{i}}:Y_i\to Y$ is
a bijection for each $i$. Assume that $z\in Y_1$. Then by Lemma~\ref{intint}
$f(Z)=Y$ and $Z=Y_1$. Hence $X$ and $\tau(X)$ contain the disjoint, internal composants
$Y_1$ and $\tau(Y_1)=Y_2$, respectively.

Let us show that if $Q$ is any internal composant of $X$, then $Q\cap Y_2=\0$.
Indeed, suppose otherwise. Then by symmetry the internal composant $\tau(Q)$ of
$\tau(X)$ intersects $Y_1$. Choose a point $u\in \tau(Q)\sm X$ and a point
$v\in \tau(Q)\cap Y_1$; then choose continuum $L\subset \tau(Q)$ which contains
both $u$ and $v$. Since $Y_1$ is an internal composant of $X$, this implies
that $\tau(Q)$ intersects all composants of $X$. On the other hand,
$f(\tau(Q))=f(Q)$ is an internal composant of $X$ by Lemma~\ref{intint}. Thus,
images of all composants of $X$ are non-disjoint from the composant $f(Q)$
contradicting Lemma~\ref{all1}.

In order to finish the proof it suffices to show that $\tau(X)\cap X\subset
E^*_X$. Suppose that this is not the case. Then $\tau(X)$ meets an internal
composant $I$ of $X$. It is easy to verify that Lemma~\ref{kras} applies to
this situation with $\tau(X)$ playing the role of $L$, $I$ playing the role of
$C$, and $X$ playing the role of $Y$.  Thus, it follows from
Lemma~\ref{kras}(\ref{cut}) that there exists a neighborhood $U$ of $\tau(X)$
and a continuum $K\subset I$ such that $K$ separates two points of $\tau(X)$ in
$U$. Since by the above $Y_2$ and $I$ are disjoint, $Y_2$ is contained in one
component of $U\sm K$, contradicting that $Y_2$ is dense in $\tau(X)$. This
completes the proof of the lemma.
\end{proof}

\section{Creating an invariant ray}\label{advaprel}

In what follows $R_\be$-essential and $R_\be$-defining crosscuts are called
simply \emph{essential} and \emph{defining}. To begin with, we need a lemma
which will allow us to simplify the applications of results of
\cite{fokkmayeovertymc07,overtymc07,kulkpink94} in this section. For simplicity
when talking of angles we often mean the points of $S^1$ with arguments equal
to these angles; here $S^1$ is considered as the boundary of the unit disk in
$\idisk$. Take the (Euclidean) convex hull $\ch(X)$ of $X$. Then the ray
$R_\be$ eventually enters $\ch(X)$ through a crosscut $\he\subset\bd(\ch(X))$
so that the tail of $R_\be$ stays inside the shadow $\sh(\he)$. Observe that
$\he$ is a straight segment. This defines the arc $I=(\hat \al, \hat \ga)$ of
angles whose rays have tails in $\sh(\he)$, and it follows that $\be\in I$.
Consider now a \emph{geodesic} $E_g$ of $\idisk$ connecting $\hal$ and $\hga$ and its
counterpart $E_f=\vp(E_g)$ which is a crosscut of $X$. Clearly, $R_\be$
eventually enters (and stays in) the shadow of $E_f$.

\begin{lemma}\label{notau}
There exists an essential crosscut $C'\subset \sh(E_f)$ with the following
properties.

\begin{enumerate}

\item We have that $\sh(C')\cap
\tau(X)=\0$ and, hence, $f(\sh(C'))\cap X=\0$.

\item Let $\A(C')$ be the set of points in $X$ accessible from $\sh(C')$.
Then $f|_{[\sh(C')\cup \A(C')]}$ is one-to-one and
    $f(\sh(C'))=\sh(f(C'))$. Moreover, there exist $\e>0, \da>0$ such that
    any defining crosscut $C_z, z\in R_\be\cap \sh(C'),$ is less than $\e$
    in diameter, and any essential crosscut $C\subset \sh(C')$ less than $\e$
    in diameter maps to a crosscut $f(C)$ which is at least $\da$-distant
    from $C$.
\end{enumerate}

\end{lemma}

\begin{proof}
(1) Consider a defining family $C_z, z\in R_\be$ of crosscuts of $\be$. Then there
exists a sequence of defining crosscuts $C_{z_i}=C_i, i=1, \dots$ such that
$z_1>_\be z_2>_\be \dots$, points $z_i\in R_\be$ converge to $X$, and
all $\ol{C}_i$'s are disjoint from $\tau(X)$ (otherwise for some $z'\in R_\be$ and
all $z<_\be z'$ we would have $C_z\cap \tau(X)\ne \0$ implying that
$\tau(X)\supset \Pi(R_\be)=X$, a contradiction with Lemma~\ref{emptyint}).
We may assume that $C_1\subset \sh(E_f)$.

Denote by $U_i$ the open component of $\C\sm [C_i\cup C_{i+1}\cup X]$ which
contains points of $R_\be$ located between $z_i$ and $z_{i+1}$. By way of
contradiction (and refining if necessary the sequence $C_i$) we may assume that
every $U_i$ contains points of $\tau(X)$. Choose a point $x\in U_{i-1}\cap
\tau(X)$ and a point $y\in U_{i}\cap \tau(X)$. Connect these points with an
arc $A$ in $U_{i-1}\cup U_i$ which intersects $C_i$ at just one point $w$. Then
choose points $s, t\in A\cap \tau(X)$ so that the subarc $B$ of $A$ with the
endpoints $s, t$ is disjoint from $\tau(X)$ and contains $w$.

It follows that $B$ is a crosscut of $\tau(X)$. Denote the endpoints of $C_i$
by $y', y''$. Also, denote by $W$ the shadow of $B$ in the sense of
$\tau(X)$. Then it follows that one of the points $y', y''$ belongs to $W$ and
the other one does not. Now, consider an internal composant $Z$ of $X$. It has
points close to both points $y'$ and $y''$, hence it has points both inside $W$
and outside $W$. However, by Lemma~\ref{emptyint} $\tau(X)$ is disjoint from
any internal composant of $X$, a contradiction. Hence indeed there exists an
essential crosscut $C$ such that $\sh(C)\cap \tau(X)=\0$ which
implies that $f(\sh(C))\cap X=\0$.

(2) We may assume that $C$ and \emph{all} $R_\be$-defining crosscuts $C_z, z\in
R_\be\cap \sh(C)$ are sufficiently small. By continuity and since $T(X)$ is
fixed point free, images of all these crosscuts are disjoint from the crosscuts
themselves (each crosscut moves off itself by a bounded away from $0$
distance). Moreover, by (A2) and because of the properties of crosscuts,
$C_z$'s approach all points of $X$. Choose a crosscut $C'$ among them so that
$C'$ is sufficiently far from $c$ and hence $f(C')$ is sufficiently far from $v$
so that $v\nin T(f(C'))$. By Lemma~\ref{pull} then $C'$ is a homeomorphic
pullback of $f(C')$ and so $f(C')$ is a small crosscut too. Also, we can choose
$C'$ so that $c\nin \sh(C')$.

We claim that $f(\sh(C'))=\sh(f(C'))$ and $f|_{\sh(C')}$ is a homeomorphism.
Indeed, $\bd f(\sh(C'))\subset f(\bd \sh(C'))$ since $f$ is open. Hence we see
that $\bd f(\sh(C'))\subset X\cup f(C')$. Points of $\sh(C')$ cannot
be mapped to $U_\iy(X)$ outside $\sh(f(C'))$ because otherwise there will be
points of $\bd f(\sh(C'))$ not in $X\cup f(C')$. Considering points close to
$C'$ shows that some points of $\sh(f(C'))$ are in $f(\sh(C'))$. Now the fact
that $f(\sh(C))\cap X=\0$ implies that $f(\sh(C'))\subset \sh(f(C'))$. Finally,
if $f(\sh(C'))\ne \sh(f(C'))$ then there will have to be points of $\bd
f(\sh(C'))$ in $\sh(C')$, a contradiction to $\bd f(\sh(C'))\subset X\cup
f(C')$. Thus, $f(\sh(C'))=\sh(f(C'))$. Hence by Lemma~\ref{pull} $f|_{\sh(C')}$
is a homeomorphism.

This easily implies that $f|_{[\sh(C')\cup\A(C')]}$ is one-to-one too. Indeed,
suppose that $z=f(x)=f(y)$ for $x\ne y\in\A(C')$. Clearly, $x\ne c, y\ne c,
z\ne v$. Choose an arc $A$ joining $x$ to $y$ in $\sh(C')$. Then $f(A)$ is a
simple closed curve $S$ so that $S\cap X=\{z\}$. Choose a small arc
$Q=[w,w']\subset S$ (in the circular order on $S$) so that $z\in Q, z\ne w,
z\ne w'$. Then there exist arcs $Q_x$ and $Q_y$, containing $x$ and $y$,
respectively such that $f(Q_x)=f(Q_y)=Q$. Since $X$ is indecomposable and,
hence, contains no cutpoints, $Q_x\cup Q_y\subset \sh(C')\cup \{x\}\cup \{y\}$.
This contradicts the fact that $f$ is one-to-one on $\sh(C')$ and completes the
proof.
\end{proof}

Recall that $\vp:\disk^\iy\to U_\iy(X)$ is a Riemann map with $\vp(\iy)=\infty$
(we extend $\vp$ over the set of angles with landing rays). Next we introduce a
construction from \cite{fokkmayeovertymc07} simplified in our case thanks to
Lemma~\ref{notau}. Take a closed round ball $B$ such that $\Int(B)\cap X=\0$
and $\ch(B\cap X)\subset \sh(E_f)\cap X$  is non-degenerate. Observe that
points of $B\cap X$ are accessible. Call a ball $B$ \emph{essential} if
$\bd(B)\sm X$ contains an essential crosscut (i.e., if $B$ ``crosses over
$R_\be$'' from one ``side'' of $R_\be$ to the other in an essential way).

Let $\Bc'$ be the family of all closed round balls $B$ with $\Int(B)\cap X=\0$
and let $\Bc$ be the family of all balls $B\in \Bc'$ such that $\ch(B\cap
X)\subset \sh(E_f)$ consists of at least two points. Then $B$ is maximal by
inclusion among all balls in $\Bc'$. It is easy to give examples reflecting
various possibilities for the sets $B\cap X$; in exceptional cases, the set
$B\cap X$ could be infinite, and in truly exceptional cases it can even contain
arcs of $\bd B$. The following lemma allows us to introduce the exact shadow in
the plane in which we will change the map to make a tail of $R_\be$ invariant.
It will be improved later and is needed here as matter of convenience to
simplify the forthcoming construction.

\begin{lemma}[\cite{fokkmayeovertymc07}]\label{gates}
There exists an essential ball $B^*\in \Bc$ of diameter less than $\e$ such that
$\Int(B^*)\subset \sh(C')$ and $|B^*\cap X|=2$.
\end{lemma}

Consider the two crosscuts which are components of $\bd(B^*)\sm X$. Choose among
them \emph{the} crosscut $\tc$ which gives rise to the shadow containing
$\Int(B^*)$. Clearly, $\tc$ is essential (it suffices to consider the picture
in the uniformization plane). Suppose that $\tc$ has the endpoints $a, d\in X$.
For the corresponding angles in $S^1$ we use the notation $\al', \ga'$. Connect
$\al', \ga'$ with a hyperbolic geodesic in $\idisk$ and denote this new
crosscut of $\idisk$ by $C'_g$. Denote the crosscut $\vp(C'_g)$ of $X$ by $C'_f$ ($C'_f$
replaces the crosscut $C'$  previously introduced in Lemma~\ref{notau} and has
all the properties of $C'$ listed in Lemma~\ref{notau}).

Now, mimicking the construction from \cite{fokkmayeovertymc07} we transport
the map $f$ to the set $U=\sh(C'_g)$ by considering a map
$g(x)=\vp^{-1}\circ f\circ \vp(x), x\in U$. Since $f|_{\sh(C'_f)}$ is a
homeomorphism, so is the map $g|_U$. Consider the arc of $S^1$ defined as
$[\al', \ga']=\bd U\cap \disk^\iy$. It follows from the construction that
$\be\in (\al', \ga')$.

To make a distinction, we use ``$g$-'' in the names of the objects in the
uniformization plane (mostly these objects are $\vp^{-1}$-images of their
counterparts from the $f$-plane). Thus, the uniformization plane is called the
\emph{$g$-plane}; to each essential crosscut $C\subset \sh(C'_f)$ we associate
its counterpart $\vp^{-1}(C)$, called an \emph{essential $g$-crosscut} (which
is an arc connecting two points of $S^1$, one in $(\al', \be)$ and the other in
$(\be, \ga')$, inside $U$ and intersecting $\vp^{-1}(R_\be)=R'_\be$ only once);
etc. It will follow that $g(\vp^{-1}(C))$ is again an essential \emph{$g$-crosscut}
associated to the crosscut $f(C)$. The results of \cite{fokkmayeovertymc07}
give more information about how the map $g$ acts on $g$-crosscuts. Namely, the
following theorem holds.

\begin{theorem}[Theorems 6.5 and 9.1, \cite{fokkmayeovertymc07}]\label{g}
The map $g$ can be continuously extended over the arc
$[\al', \ga']$. Moreover, it has the following properties:

\begin{enumerate}

\item $g(\be)=\be$;

\item $g$ maps the arc $[\al', \ga']$ onto the arc $[g(\al'), g(\be')]$
    homeomorphically and changes orientation (so that $g$ \emph{flips}
    essential $g$-crosscuts contained in $\sh(C'_g)$);

\item for every essential crosscut $C''\subset \sh(C'_f)$ of diameter less
    than $\e$, the $g$-crosscut $g(\vp^{-1}(C''))$ separates  $\vp^{-1}(C'')$
    from $\iy$ in $\idisk$.

\end{enumerate}

\end{theorem}


Suppose that a closed set $Y\subset S^1$ is chosen and consider its convex hull
$\HCD(Y)=A$ in the sense of the hyperbolic metric in $\idisk$. Hence $\HCD(Y)$
can be obtained by considering the set of components $C_i$ of $S^1\sm Y$ and
joining the endpoints $a_i,b_i$ of $C_i$ by the geodesic in the hyperbolic
metric (i.e., the intersection of the  round circle through the points $a_i$
and $b_i$ with $\idisk$ which crosses $S^1$ perpendicularly; see
\cite{fokkmayeovertymc07}). Technically, $A\cap\idisk$ is a subset of $\idisk$
which can be mapped to the $f$-plane by $\vp$. The closure of
$\vp(A\cap\idisk)$ may be very complicated and not homeomorphic to $\ol{A}$.
However if the rays with the arguments from $Y$ land at distinct points of $X$
then $\ol{A}$ and $\ol{\vp(A)}$ are homeomorphic. In this case the set $A$ will
be called a \emph{$g$-cell} and the set $\vp(A)$ will be called an
\emph{$f$-cell}. In what follows speaking of a map $\vp$ restricted onto a
cell, we \emph{always extend $\vp$ over the boundary of the cell} (by
definition, $\vp$ then remains a homeomorphism). Observe that a $g$-cell could
be an arc (without endpoints) or a Jordan disk (with points on $\bd(\idisk)$
removed). Moreover, if $\ta\in S^1$ is such an angle that $R_\ta$ lands on
$\vp(\ta)$ then we say that $\ta$ is a \emph{degenerate $g$-cell} and
$\vp(\ta)$ is a \emph{degenerate $f$-cell}.

The map $\vp$ can give a good correspondence between closed Jordan disks in the
$f$-plane and in the $g$-plane. Suppose that $D$ is a closed Jordan disk in the
$g$-plane such that $\vp$ extends over the boundary of $D$. Then we call $D$
\emph{strongly homeomorphic} to $\vp(D)=A$ if $|D\cap S^1|\ge 2$ and $\vp|_D$
is a homeomorphism. A closed Jordan disk $A$ in the $f$-plane, strongly
homeomorphic to $\vp^{-1}(A)$, is called \emph{admissible}. One can transform
an admissible Jordan disk $A$ in the $f$-plane to an $f$-cell $H(A)$: choose
the hyperbolic convex hull $\HCD(\ol{\vp^{-1}(A)}\cap S^1)$ of $\ol{D}\cap S^1$
and then take its $\vp$-image denoted $H(A)$.

A hyperbolic geodesic $\ell$ in the $g$-plane is called a \emph{$g$-geodesic}.
If $\vp(\ol{\ell})$ is homeomorphic to $\ol{\ell}$ (i.e., to the closed
interval), then $\vp(\ol{\ell})$ is called an \emph{$f$-geodesic}.
Thus, boundary arcs of a $g$-cell are $g$-geodesics whose $\vp$-images are
$f$-geodesics. Clearly, two angle-arguments give rise to a $g$-geodesic whose
$\vp$-image is an $f$-geodesic if and only if the rays with these arguments
land. We will also consider a \emph{degenerate geodesic with argument $\ta$},
i.e. an accessible point at which the ray $R_\ta$ land. Important facts
concerning $f$-geodesics were established in \cite{overtymc07}.

\begin{definition}\label{topdi}

Given an admissible Jordan disk $A$ in the $f$-plane, define the following
sets:

(a) $P_A=(\bd A)\cap X$;

(b) $C_A=A$ (if $A$ is a crosscut) or the unique component of $(\bd A)\sm X$ which
is a crosscut with $\Int(A)\subset \sh(C_A)$;

(c) $I_A\subset S^1$ is an arc such that rays with the
arguments from $I_A$ have tails in $\sh(C_A)$;

(d) $Y_A=\ol{\vp^{-1}(A)\cap S^1}$ (recall, that according to the definition
we have that $\vp^{-1}(A)\cap\idisk=\HCD(Y_A)$).

\end{definition}

We need to study $g$-geodesics and $f$-geodesics. As a tool we use the
so-called \emph{maximal ball} foliation constructed and studied in
\cite{fokkmayeovertymc07,overtymc07,kulkpink94}. Given a crosscut $T$, denote
the set $\sh(T)\cup T$ by $\sh^+(T)$. We foliate the sets $\sh^+(C'_g)$ in the
$g$-plane and $\sh^+(C'_f)$ in the $f$-plane by  corresponding (to each other)
and specifically constructed $g$-cells and $f$-cells. However first we need to
introduce some definitions.

\begin{definition}\label{folia}
Suppose that there is a $g$-geodesic $T$ in the $g$-plane and that $\vp(T)$ is
an $f$-geodesic. Suppose that there is a family $\A$ of $g$-cells such that the
following holds. For \emph{any} $z\in \sh^+(T)$ either there is a
\emph{unique} $g$-geodesic from the boundary of a $g$-cell $A\in \A$ containing
$z$, or $z$ belongs to the interior of a unique $A\in \A$. The geodesic
containing $z$ maybe the intersection of two distinct $g$-cells. Then we call
$\A$ a \emph{($g$-)foliation} (of $\sh^+(T)$).

By the definition of a $g$-cell this property is transported to the $f$-plane
by means of the map $\vp$ and applies also to the family $\vp(\A)$ of the
corresponding $f$-cells and the set $\sh^+(\vp(T))$. Then $\vp(\A)$ is called a
an \emph{($f$-)foliation} (of $\sh^+(\vp(T))$). The collections $\A, \vp(\A)$
are then said to be \emph{sibling foliations} (of $\sh^+(T)$ and
$\sh^+(\vp(T))$, respectively). Also, the closure of a $g$- or an $f$-geodesic on
the boundary of a $g$- or an $f$-cell is called a \emph{($g$-) or ($f$-)leaf}
(of the corresponding foliation).
\end{definition}

In some cases we can say much more about continuity properties of foliations.
To this end we need another definition.

\begin{definition}\label{upper1}
Suppose that there is a $g$-geodesic $T$ in the $g$-plane and that there is a
foliation $\A$ of $\sh^+(T)$. Then $\A$ is said to be
\emph{upper-semicontinuous} if the following holds: if a sequence of distinct
$g$-cells from $\A$ converges, and its limit is not a point, then it converges
to a $g$-leaf. If the $\vp$-images of these $g$-cells converge to the $f$-leaf
$\vp(\ell)$, then we say that $\vp(A)$ is \emph{upper-semicontinuous}.
\end{definition}

The sheer fact that $\A$ is a foliation of $\sh^+(T)$ easily implies that $\hg$
is upper-semicontinuous. Indeed, a sequence of $g$-cells $A_i$ which does not
converge to a point must converge to a $g$-geodesic $\ell$ (recall that as
elements of $\A$ the $g$-cells do not intersect inside $\idisk$). If $\ell$ is
not from the boundary of an element of $\A$, then, since all points inside
$\sh^+(T)$ must belong to an element of $\A$, we will find an element of $\A$
which intersects $A_i$, a contradiction. The fact that the corresponding
foliation $\vp(\A)$ of $\sh^+(\vp(T))$ is upper-semicontinuous is rather
non-trivial.

We also give a definition close to Definition~\ref{upper1} which deals with
\emph{null-sequences} of admissible Jordan disks. Given an admissible Jordan
disk $Z$, define $\rho_Z(x, y)$ as the infimum of diameters of open arcs $J(x,
y)=J$ in $U_\iy$ whose closures $\ol{J}$ connect $x, y$ and are homotopic to
arcs $I(x,y)=I\subset Z$ connecting $x, y$ inside $Z\cap U_\iy$ under a
homotopy in $U_\iy$ fixing $x$ and $y$ (the definition is inspired by the
\emph{Mazurkiewicz metric} discussed later).

\begin{lemma}\label{conthyp}
There exists a universal constant $K$ with the following property. Suppose that
$A$ is an admissible Jordan disk in the $f$-plane such that for any two points
$x, y$ of $\ol{A}\cap X=P_A$ we have $\rho_A(x, y)\le M$. Then $\diam(H(A))\le
KM$.
\end{lemma}

\begin{proof}
Take points $x, y\in P_A$. By the assumptions of the lemma and
by Pommerenke \cite[Theorem 4.2]{pom92}, there exists a
constant $K'$ such that the diameter of the $f$-geodesic $G(x,
y)$ is at most $K'M$. The geodesic $G(x,y)$ is chosen to be
homotopic to one of the arcs $I(x,y)=I$ connecting $x, y$
inside $Z\cap U_\iy$ under a homotopy in $U_\iy$ fixing $x$ and
$y$. Now, given two points $u, v\in H(A)$, connect them with
the $f$-geodesic $I$ and then extend it to the $f$-geodesic
$I_1$ which connects points $u\in \bd H(A)$ and $v\in \bd
H(A)$. Choose the $f$-geodesics $I_2, I_3$ from the boundary of
$H(A)$ such that $u\in I_2$ and $v\in I_3$. Fix an endpoint of
$I_2$, an endpoint of $I_3$, and the $f$-geodesic $I_4$ which
connects them. Then by the triangle inequality and by the above
$|I|\le |I_1|\le |I_2|+|I_3|+|I_4|\le 3K'M$. So, any two points
of $H(A)$ can be connected with an $f$-geodesic of diameter at
most $3K'M$ which completes the proof.
\end{proof}

In what follows by a \emph{null-sequence} we mean that of sets
whose diameters converge to zero.

\begin{definition}\label{nul1}
Suppose that $\A$ and $\vp(\A)$ are sibling foliations. Suppose that given a
sequence of $g$-cells $A_i\in \A$ and the corresponding  sequence  $\vp(A_i)$
of $f$-cells, we have that \emph{$\{A_i\}$ is a null-sequence if and only if
$\{\vp(A_i)\}$ is a null-sequence}. Then we call the foliations $\A, \vp(\A)$
\emph{null preserving}.
\end{definition}

We can finally single out pairs of foliations we want to work with.

\begin{definition}\label{upper}
Suppose that there is a crosscut $T$ in the $g$-plane and that $\vp(T)$ is a
crosscut of $X$. Moreover, suppose that $\A, \vp(A)$ are sibling foliations of
$\sh^+(T)$ and $\sh^+(\vp(T))$ respectively which are upper-semicontinuous and
null preserving. Then we say that they form a \emph{canonic pair (of
foliations)} or just that they are \emph{canonic}.
\end{definition}

Observe that elements of foliations are always $g$-cells and $f$-cells.
If $\A, \vp(\A)$ are canonic then for distinct $g$-cells $A', A''\in \A$ only
the following cases are possible: (1) $A', A''$ have a common $g$-geodesic
in their boundaries but
are otherwise disjoint: (2) the closures of $A', A''$ have a unique point of
$S^1$ in common; (3) the closures of $A', A''$ are disjoint.  This implies that
the sets $\ol{A'}\cap S^1, \ol{A''}\cap S^1$ are \emph{unlinked}, i.e. one of
them is contained in a component of the complement to the other (except for the
endpoints). Observe, that the collection of geodesics in the boundaries of all
the $g$-cells $A\in\A$ is a \emph{lamination} of  $\sh(T)$  (cf. \cite{thu85}).

In general $\vp$ is far from being continuous near $S^1$. However the existence
of canonic foliations allows us to use a version of continuity of $\vp$
justified by Lemma~\ref{continu}. By $d(a, b)$ we understand the standard
Euclidian distance between two points $a, b\in \C$.

\begin{lemma}\label{continu}
Let $\A, \vp(\A)$ be canonic foliations. Then the family of restrictions
$\{\vp|_{\ol{A}}\}_{A\in \A}$ is \emph{equicontinuous}: for any $\e'$ there exists
$\da'$ such that if $x', z'\in \ol{A}$, $ A\in \A$ and $d(x', z')\le \da'$ then
$d(\vp(x'), \vp(z'))\le \e'$.
\end{lemma}

\begin{proof}
By way of contradiction suppose that $x'_i, z'_i\in \ol{A_i}, A_i\in \A,$ are
two sequences with $d(x'_i, z'_i)\to 0$ (here $A_i\in \A$ is a sequence of
$g$-cells from $\A$), but $d(\vp(x'_i), \vp(z'_i))\not \to 0$. Refining the
sequences, we may assume that $\lim x'_i=\lim z'_i=a$ while $\vp(x'_i)\to x_f,
\vp(z'_i)\to z_f$ and $x_f\ne z_f$. Let us show that $A_i$ is not a
null-sequence. Indeed, if $A_i$ is a null-sequence, then $\vp(A_i)$ is a
null-sequence (since $\A, \vp(A)$ are canonic) and $d(\vp(x'_i), \vp(z'_i))\to
0$, a contradiction. Since the foliations are canonic we may assume that $A_i$
converge to a $g$-geodesic $\ell$ on the boundary of a $g$-cell. Then
$\vp(A_i)$ converge to the $f$-geodesic $\vp(\ell)$, and hence the fact that
$\lim x'_i=\lim z'_i=a$ implies that $\lim \vp(x'_i)=\lim \vp(z'_i)$, a
contradiction with $x'_f\ne z'_f$.
\end{proof}

An important example of a canonic pair of foliations comes from the
construction of the family $\Bc$ of maximal closed, round balls introduced
right before Lemma~\ref{gates} (the corresponding tools can be found in
\cite{fokkmayeovertymc07} and \cite{kulkpink94}). Let $\hg$ be the family of
all sets given by formulas $\HCD(\vp^{-1}(B\cap X))\subset \sh(C'_g),$ $B\in
\Bc$. Let the corresponding family of all sets $H(B)\subset \sh(C'_f), B\in
\Bc$ by $\hx$.

\begin{theorem}[\cite{fokkmayeovertymc07}, see also \cite{kulkpink94}]\label{foli0}
The foliations $\hg$ and $\hx$ are sibling foliations of $\sh^+(C'_g)$ and $\sh^+(C'_f)$
respectively.
\end{theorem}

The main next step is to show that $\hg, \hx$ form a canonic pair of
foliations. The following corollary follows from \cite{fokkmayeovertymc07}.

\begin{corollary}\label{foli1} Suppose that there is a sequence $A_n$
of distinct $g$-cells from $\hg$. Then the following facts hold.

\begin{enumerate}

\item
If $A_n$ converge to a hyperbolic geodesic $\ell=\al\be$, then $\ell$ is a
$g$-leaf of $\hg$, $\vp(\ell)$ is an $f$-leaf of $\hx$, and $\ol{\vp(A_n)}$ converge
to $\ol{\vp(\ell)}$.

\item
If $A_n$ is a null-sequence, then $\vp(A_n)$ is a null-sequence.

\end{enumerate}

\end{corollary}

\begin{proof}
Each of the $f$-cells $\vp(A_n)$ corresponds to a maximal ball $B_n\in\Bc$. In
case (1) the balls $B_n$ have bounded away from zero diameters and may be
assumed to converge to a non-degenerate maximal ball $B$. In case (2), by way
of contradiction and after refining the sequence, we may assume that $A_n$
converge to a point $\al=\be\in S^1$ while $\vp(A_n)$ converge to a
non-degenerate continuum. By Lemma~\ref{conthyp} the fact that the
diameters of $\vp(A_n)$ are bounded away from zero implies that diameters of
$B_n\cap \vp(A_n)$ are bounded away from zero (otherwise by Lemma~\ref{conthyp}
the diameters of $\vp(A_n)$ go to zero). Hence again the balls
$B_n$ converge to a maximal ball $B\in\Bc$. By Lemma 4.2 in
\cite{fokkmayeovertymc07}, $\vp(A_n)\subset B_n$. By Lemma 5.1 in
\cite{fokkmayeovertymc07} $\lim \vp(A_n)\cap X=\{\al, \be\}$. Hence $\lim
\ol{\vp(A_n)}=T\subset B$. So far the situation has been similar for both
cases, however now we consider them separately.

(1) From geometric considerations in the $f$-plane $T\subset
\vp(\ell)\cup\imp(\al)\cup \imp(\be)$. The fact that $\vp(A_n)\cap X$ converges
to $\{\vp(\al),\vp(\be)\}$ and $\Int(B_n)$ does not contain $\vp(\al),\vp(\be)$
implies that there is an arc $J_n\subset \bd(B)$ such that
$J_n\subset\Int(B_n)$ (thus, $J\cap X=\0$) and $\lim J_n=J$ is an arc in
$\bd(B)$ from $\al$ to $\be$. Hence $J\cap X=\{\vp(\al),\vp(\be)\}$. Let
$K=\bd(B)\sm J$ and suppose that $K$ contains a point $x\in X\sm\{\vp(\al),
\vp(\be)\}$. Let $C_n$ be the component of $B_n\sm\ol{\vp(\ell)}$ which
contains $\vp(A_n)$ (recall that $\vp(\al)$ and $\vp(\be)$ are outside the
interior of $B_n$). Then it is easy to observe that there exists $\e>0$ such
that $d(x, C_n)>\e$. Hence $\lim \vp(A_n)=\ol{\vp(\ell)}$.

(2) By the above and by the conditions of the lemma $T\subset B\cap X$. Since
$T$ is a continuum, it follows that $T\subset \bd B$ is a non-degenerate arc.
Clearly, this contradicts the act that $A_n$ is a null-sequence.
\end{proof}

In Theorem~\ref{foli} we summarize the above results.

\begin{theorem}[\cite{fokkmayeovertymc07}, see also \cite{kulkpink94}]\label{foli}
The sibling foliations $\hg$ and $\hx$ form a canonic pair.
\end{theorem}

\begin{proof}
By the remark after Definition~\ref{upper1}, $\hg$ is upper-semicontinuous.
Suppose that $\vp(A_n)$ is a convergent sequence of distinct $f$-cells whose
limit is not a point. By Corollary~\ref{foli1}(2) the diameters of $A_n$ are
bounded away from zero. Hence we may assume that $A_n$ converge to a $g$-leaf
of $\hg$. Now by Corollary~\ref{foli1}(1) $\vp(A_n)$ converge to an $f$-leaf of
$\hx$ and $\hx$ is upper-semicontinuous. On the other hand, if $\vp(A_n)$ is a
null-sequence of $f$-cells in $\hx$ while $A_n$ is not a null-sequence of
$g$-cells, then, after refining $A_n$ and by Corollary~\ref{foli1}(1), we may
assume that $A_n$ converge to a $g$-leaf $\ell$ and $\vp(A_n)$ converge to its
non-degenerate image $\vp(\ell)$, a contradiction.
\end{proof}

Let us state a useful corollary of results of \cite{fokkmayeovertymc07}.

\begin{corollary}\label{useful} Let $\xi\in (\al', \ga')\subset S^1$ be an
angle which is not an endpoint of a $g$-leaf of $\hg$. Then there exists a
sequence $B_i$ of $R_{\xi}$-essential balls from $\Bc$ with radii converging to
zero and such that $\sh(C'_f)\supset \sh(C_{H(B_1)})\supset
\sh(C_{H(B_2)})\supset \dots$ and $\HCD(\vp^{-1}(B_i\cap X))\to \{\xi\}$.
\end{corollary}

By the above, the foliations $\hg$ and $\hx$, induced by maximal balls, form a
canonic pair of foliations. The next lemma allows us to extend this result to a
more general situation. However first we need to define the \emph{Mazurkiewicz
metric} $\rho$. Consider the set $G(X)=G$ of all pairs $(\vp(\al), \al)$ such
that ray $R_\al$ lands on $\vp(\al)$. Given $t=(\vp(\al), \al)$ and
$w=(\vp(\be), \be)$, define $\rho(t, w)$ as follows. If $t=w$ set $\rho(t,
w)=0$. If $t\ne w$, let $\rho(t, w)$ be the infimum of the diameters of open
arcs $J$ in $U_\iy$ whose closures are closed arcs from $\vp(\al)$ to
$\vp(\be)$ or simple closed curves containing $\vp(\al)=\vp(\be)$ such that
$\vp^{-1}(J)$ is homotopic to the geodesic $\al\be\in \idisk$ under a
homotopy fixing $\al, \be$.

\begin{lemma}\label{cellcont}
Let $T$ be a $g$-geodesic and $\vp(T)$ be an $f$-geodesic. Let $\A, \vp(\A)$ be
sibling foliations of\, $\sh^+(T), \sh^+(\vp(T))$ respectively. Consider a
sequence $A_i$ of distinct $g$-cells of $\A$ which converge to a (degenerate)
$g$-leaf $\ell$ such that $\ol{\vp(A_i)}\cap X\to \ol{\vp(\ell)}\cap X$ in the
Mazurkiewicz metric. Then the closed $f$-cells $\ol{\vp(A_i)}$ converge to the
closed (degenerate) $f$-leaf\, $\ol{\vp(\ell)}$.
\end{lemma}

\begin{proof}
Consider first the case when $\ell=\ta$ is a degenerate $g$-geodesic. By the
conditions of the lemma it follows that Lemma~\ref{conthyp} applies to
$\vp(A_i)$. Thus, the diameters of $\vp(A_i)$ converge to zero and $\vp(A_i)$
converge to $\vp(\ta)$ as required.

Suppose now that $\ell=\al\be$ and $\al\ne \be$. Clearly, $\vp(\ell)\subset
\limsup \vp(A_i)$. Partition the sequence $A_i$ into two sequences which
converge to $\ell$ from opposite sides (in $\idisk$) and consider them
separately. Without loss of generality assume that there are $\al_i, \be_i\in
\ol{A_i}\cap S^1$ such that $\al_i<\al<\be<\be_i$ and there are no points of
$A_i\cap S^1$ between $\al_i$ and $\al$ and between $\be$ and $\be_i$. Let
$\al_i\al$ be the $g$-geodesic connecting $\al_i$ and $\al$. Define the
\emph{left wing} $\imp^-(\al)=\bigcap \ol{\sh(\vp(\al\al_i))}$ of the
impression $\imp(\al)$; define the \emph{right wing} $\imp^+(\be)$ of the
impression $\imp(\be)$ similarly.  Then it is easy to see that $\limsup
\vp(A_i)\subset \vp(\ell)\cup \imp^-(\al)\cup \imp^+(\be)$.

It follows from the conditions of the lemma and Theorem 4.2 from \cite{pom92}
that $\vp(\al_i\al)\to \vp(\al)$ and $\vp(\be_i\be)\to \vp(\be)$. Choose
$\e>0$. Choose $i$ such that for any $n\ge i$ we have the following: (1)
$\diam(\vp(\al_i\al))+\diam(\vp(\be\be_i))<\e$, and (2) denoting the set
$\vp(\al_i, \al)\cap A_n$ by $A^\al_n$, we have that the $\rho$-distance
between the set $\vp(A^\al_n)$ and $\vp(\al)$ is less than $\e$ and similarly
for $\be$ and the similarly defined sets $A^\be_n$. The set $A_n\sm [\al\al_i
\cup \be\be_i]$ consists of three components. Let $T_n$ be the one of them
containing $A^\al_n$, $L_n$ be the one of them containing $A^\be_n$, and $M_n$
be the remaining third component. Then $\lim M_n=\ell$ and $\lim
\vp(M_n)=\vp(\ell)$.

Consider the hyperbolic convex hull $Q^\al_n$ of the points of $A^\al_n, \al_i,
\al$. Then $T_n\subset Q^\al_n$. By the choice of $i$ any two points of
$\vp(Q^\al_n)\cap X$ can be joined by an arc in $U_\iy$ of diameter less than
$2\e$. Then by Lemma~\ref{conthyp} $\diam(\vp(T_n))\le
\diam(\vp(Q^\al_n))\le 2K\e$. Hence $\lim \vp(T_n)=\{\vp(\al)\}$. Similarly,
$\lim \vp(L_n)=\{\vp(\be)\}$. Thus, $\lim \vp(A_n)=\vp(\ell)$ as desired.
\end{proof}

Observe that Lemma~\ref{cellcont} gives another proof of Corollary~\ref{foli1}.
Indeed, let $A_n$ be a sequence of distinct $g$-cells of $\hg$ taken
from Corollary~\ref{foli1}. As in the initial part of the proof of that
corollary, we may assume that $A_n$ converge to a $g$-leaf $\al\be$ and
the corresponding maximal balls $B_n$, associated with
$f$-cells $\vp(A_n)$, have bounded away from zero radii and converge to a non-degenerate
ball $B$. Then it follows from the geometric considerations that the $\rho$-distance
between two ``clusters'' of $\vp(A_n)\cap B_n=\vp(A_n)\cap X$ and the corresponding
points $\al$ or $\be$ goes to zero. By Lemma~\ref{cellcont}, this implies that
$\ol{\vp(A_n)}\to \ol{\vp(\al\be)}$ as desired.

Denote by $\mathcal{G}$ the family of all $g$-leaves in $\hg$; denote the union
of all such leaves by $\G^*$. Note that if $B\in \Bc$ and $|P_B|=|B\cap X|=2$,
then $\HCD(\vp^{-1}(B\cap X))\in \G$ and $\HCD(\vp^{-1}(B\cap X))$ is a
$g$-leaf. A set $\HCD(B)$ which is not a $g$-leaf is said to be a
\emph{$g$-gap}. Thus, if $|P_B|=|B\cap X|\ge 3$, then $\HCD(\vp^{-1}(B\cap X))$
is a $g$-gap and all geodesics in its boundary are $g$-leaves The $\vp$-images of
$g$-leaves are called \emph{$f$-leaves}. Their entire family is denoted by
$\G_f$. The $\vp$-images of $g$-gaps are called \emph{$f$-gaps.} The
$\vp$-image of $\G^*$ is denoted by $\G^*_f$.

We will show in Theorem~\ref{hom} that we can modify the map $f$ to a branched
covering map $f^*$ such that
$f^*|_{R_\beta\cap \sh(C)}:R_\beta\cap \sh(C)\to R_\beta$ is an embedding for
some $R_\beta$-essential crosscut $C$.

\begin{theorem}\label{hom}
There exists a $R_\beta$-essential crosscut $C$ and a branched covering map
$f^*:\C\to\C$ of degree $-2$ such that  $f^*|_{{\small \C}\sm \sh(C)}=f|_{\C\sm
\sh(C)}$, $f^*|_{\sh(C)}:\sh(C)\to f(\sh(C))$ is a homeomorphism with
$f^*(R_\beta\cap \sh(C))=R_\beta\cap \sh(f(C))$, and there exists $x_0\in
R_\beta$ such that for all $x<_\be x_0$,  $f^*(x)>_\be x$.
\end{theorem}

\begin{proof}
The first step in the proof is to refine the sequence $B_i$ from
Corollary~\ref{useful} so that it has the following property. Consider $H(B_1)$
and the crosscut $C_{H(B_1)}$, ``farthest away from $X$'' among crosscuts from
$\bd H(B_1)$. Let $I_{H(B_1)}=I_{B_1}=(\al, \ga)\subset S^1$. By
Corollary~\ref{useful}, we may choose $B_1$ so that the following holds. Denote
the endpoints of $C_{H(B_1)}$ by $p_1$ and $q_1$. Choose $B_1$ so that both
$C_{H(B_1)}$ and $f(C_{H(B_1)})$ are contained in $\sh(C'_f)$. In addition
choose it so that the $f$-geodesic joining $f(p_1), f(q_1)$ is also contained
in $\sh(C'_f)$.

Let $Q=\sh(C_{H(B_1)})$. It will also be useful to consider $\vp^{-1}(Q)$.
Clearly, $\bd \vp^{-1}(Q)$ is the union of $\vp^{-1}(C_{H(B_1)})$ and the arc
$I_{B_1}=(\al, \ga)\subset S^1$. By Theorem~\ref{foli}
the restrictions of the foliations $\hg$ and $\hx$ to $\vp^{-1}(Q)$ and $Q$
form a canonic pair. For simplicity, we still denote them $\hg$ and $\hx$.

The idea is to modify the map $g$ so that a tail of the radial ray
$\vp^{-1}(R_\be)=r_\be$ is fixed and then transport it back to the
$f$-plane. The new map $g^*$ will coincide with $g$ on $S^1$. The modification
of $g$ takes place inside $\sh(C'_g)$ while the corresponding
modification of $f$ takes place in $\sh(C'_f)$. We construct $g^*$ in a few
steps. First, we construct a foliation whose elements are the future images
(under the map $g^*$ when it is defined) of the associated elements of $\hg$.
Then $g^*$ is defined inside elements of $\hg$ so that it satisfies the
standard continuity and extension conditions and keeps a tail of $R_\be$
invariant as desired.

Elements of the new foliation are associated to elements of $\hg$ as follows.
Observe that by Theorem~\ref{g} the map $g$ on $(\al, \ga)\subset S^1$ is a
homeomorphism with a fixed repelling point $\be$; the map $g$ flips points
around $\be$. Therefore the fact that all sets $\ol{A}\cap S^1, A\in \hg$ are
unlinked implies that the sets $g(\ol{A}\cap S^1), A\in \hg$ are unlinked too.
Denote the convex hull of the set $g(\ol{A}\cap S^1)$ in the hyperbolic metric
in $\disk^\iy$ by $Z_A$. We conclude that the sets $Z_A$ form a foliation
$\hg^*$ of the set $\sh(K)$ where $K$ is the $g$-geodesic in $\idisk$ connecting
points $g(\al), g(\ga)$.

For angles $u, v\in S^1$ let $uv$ be the $g$-geodesic connecting $u$ and $v$.
We want to define a homeomorphic extension $g^*$ from $\sh(\al\ga)$ to
$\sh(g(\al)g(\ga))$. We do this so that, for each $g$-cell $A$, $g^*$ maps
$A\in \hg$ onto $Z_A$ as an orientation preserving homeomorphism and coincides
with $g$ on $\ol{A}\cap S^1$. In particular, a boundary $g$-geodesic of $A$
maps to the corresponding boundary $g$-geodesic of $Z_A$ so that the endpoints
are mapped as the map $g$ prescribes. We can then extend this map over all
gaps by mapping the barycenter of each gap to the barycenter of the image gap
and subsequently ``coning'' the map on the gap (see \cite{overtymc07}).

Let $A\cap r_\be\ne \0$. By Theorem~\ref{g} then $Z_A$ intersects $r_\be$
farther away from $S^1$ than $A$ in the sense of the order on $r_\be$. It is
easy to see that then $g^*$ can be designed so that in addition to the above we
have $g^*(r_\be\cap A)=r_\be\cap Z_A$. By Theorem~\ref{foli} the entire tail of
$r_\be$ inside $\vp^{-1}(Q)$ is covered by the sets $A\cap r_\be, A\in \hg$,
hence the new map $g^*$ maps $\vp^{-1}(Q)$ onto $\sh(K)$ so that $g^*(r_\be
\cap \vp^{-1}(Q))=r_\be\cap \sh(K)$. Since by Theorem~\ref{g} the map $g$ is
continuous, the newly constructed map $g^*$ can be constructed to be continuous
on $\ol{\vp^{-1}(Q)}$. Clearly, $g^*$ is a homeomorphism. The next claim is
crucial for the proof of the theorem.

\vskip .1in

\noindent \textbf{Claim 1.} \emph{The foliations $\hg^*, \vp(\hg^*)$ of $\sh(K)$
are canonic.}

\vskip .1in

\noindent\emph{Proof of Claim 1.} Let $A$ be a $g$-cell. Let us show that
$\vp(\ol{Z_A})$ and $\ol{Z_A}$ are homeomorphic. Indeed, by the above $g^*|_A$
is a homeomorphism from $A$ to $Z(A)$. Since $\hg$ and $\hx$ are a canonic
pair, $A$ and $\vp(A)$ are strongly homeomorphic. By Lemma~\ref{notau}
$f(\ol{\vp(A)})$ and $\vp(\ol{A})$ are homeomorphic. This implies that $\vp(\ol{Z_A})$ and
$\ol{Z_A}$ are homeomorphic to each other and to $\ol{A}$ and $\ol{\vp(A)}$.

Hence $\hg^*, \vp(\hg^*)$ are sibling foliations.
Moreover, by the remark right after Definition~\ref{upper1}, $\hg^*$ is
upper-semicontinuous. We need to show that the foliation $\vp(\hg^*)$ is
upper-semicontinuous.

Indeed, consider a sequence of $g$-cells $A_i\in \hg^*$ which converges to a
$g$-geodesic $\ell$ in the boundary of a $g$-cell $A\in \hg^*$. Denote the
endpoints (``end-angles'') of $\ell$ by $\ta'$ and $\ta''$.
As in the remark right after Lemma~\ref{cellcont}, the geometric considerations imply
that the $\rho$-distance between two ``clusters'' of the set $\vp(A_n)\cap X$
and the appropriate points $\vp(\ta'), \vp(\ta'')$ goes to zero.
It follows from Lemma~\ref{cellcont}
that then the $f$-cells $\vp(\ol{A_i})$ converge to the $f$-geodesic in the
boundary of $\vp(\ol{A})$ connecting the landing points of the rays $R_{\ta'}$ and
$R_{\ta''}$.

It remains to show that under $\vp$ null-sequences of cells in $\hg^*$ and in
$\vp(\hg^*)=\hx^*$ correspond to each other. One way it immediately follows: if
a sequence of $f$-cells is null but their $\vp$-preimages form a sequence which
is not null, then we can refine the latter to get a sequence converging to a
non-trivial set inside $\idisk$. Its $\vp$-image has to be contained in the
limit of the just introduced $f$-cells which can only be a point, a
contradiction. Now, suppose that $A_i\in \hg^*$ is a null sequence. We may
assume that $\ol{A_i}\cap S^1=g(\ol{D_i}\cap S^1)$ where $D_i\in \hg$. Since
$g$ is a homeomorphism, $D_i$ form a null sequence. Then by Theorem~\ref{foli},
the  $f$-cells $\vp(D_i)$ form a null-sequence too. Therefore by continuity and
Lemma~\ref{notau} $\dia_a(f(\vp(D_i)))\to 0$. By Lemma~\ref{conthyp} and by the
construction then $\dia_a(\vp(A_i))\to 0$ as desired. Hence $\hg^*$ is canonic
and Claim 1 is proven. \qed

We define $g^*$ so that it maps $A\in \hg$ onto $Z_A$ as an orientation
preserving homeomorphism and coincides with $g$ on $\ol{A}\cap S^1$. In
particular, a boundary $g$-geodesic of $A$ maps   to the corresponding boundary
$g$-geodesic of $Z_A$ so that the endpoints are mapped as the map $g$
prescribes. We can then extend this maps over all gaps by mapping the
barycenter of each gap to  the barycenter of he image gap and subsequently
coning the map on the gap.  Suppose that $A\cap r_\be\ne \0$. It follows from
Theorem~\ref{g} that then $Z_A$ intersects $r_\be$ farther away from $S^1$ than
$A$ in the sense of the order on $r_\be$. It is easy to see that then $g^*$ can
be designed so that in addition to the above we have $g^*(r_\be\cap
A)=r_\be\cap Z_A$. By Theorem~\ref{foli} the entire tail of $r_\be$ inside
$\vp^{-1}(Q)$ is covered by the sets $A\cap r_\be, A\in \hg$, hence the new map
$g^*$ maps $\vp^{-1}(Q)$ onto $\sh(K)$ so that $g^*(r_\be \cap
\vp^{-1}(Q))=r_\be\cap \sh(K)$. Since by Theorem~\ref{g} the map $g$ is
continuous, the newly constructed map $g^*$ can be constructed to be continuous
on $\ol{\vp^{-1}(Q)}$.

Now the map $g^*$ can be transported to the $f$-plane by means of the map $\vp$. To
begin with, the new map $f^*$ is defined only on $Q$ as follows: $f^*=\vp \circ
g^*\circ \vp^{-1}$. Moreover, by the construction the map $f^*$ is also defined
on \emph{entire} sets $\ol{A}, A\in \hx$. Still, there are two problems which
need to be resolved before we complete the proof of the theorem.

First, we need to extend $f^*$ from $Q$, beyond the crosscut $C_{H(B_1)}$ which
serves as the gates into the shadow $Q$, onto the strip between $C_{H(B_1)}$
and $C'_f$. To see that this is possible, notice that under $f^*$ the
$f$-geodesic crosscut $C_{H(B_1)}$ is mapped so that (a)
$f^*(C_{H(B_1)})=\vp(K)$ is an $f$-geodesic crosscut in whose shadow
$C_{H(B_1)}$ is contained, and (b) $f$ and $f^*$ on $C_{H(B_1)}$ are homotopic
outside $T(X)$. Clearly, the map $f^*$ can be extended to the region between
$C_{H(B_1)}$ and $C'_f$ as a homeomorphism so that its action coincides with
that of $f$ on $\C\sm\sh( C'_f)$ and with $f^*$ on $C_{H(B_1)}$.

Second, we define $f^*$ on the entire $\C$ as $f^*$ (already defined above on
$\sh(C'_f)$) and $f$ elsewhere. We need to show that  the map $f^*$ is
continuous. This needs to be proven only at points of $X$. Indeed, let $x_i\to
x, x\in X$ and show that then $f^*(x_i)\to f(x)$. We may assume that $x_i\nin
T(X)$. Then for each $i$ we can choose an $f$-cell $L_i\in \hx$ with $x_i\in
L_i$. To each $L_i$ we associate the corresponding $g$-cell $M_i=\vp^{-1}(L_i)$
which by the construction is the convex hull of the set $\ol{M_i}\cap S^1$ in
the hyperbolic metric on $\idisk$.

Suppose that $\dia (L_i)\to 0$. Then sets $\vp^{-1}(L_i)$ form a null-sequence
(since $\hg, \hx$ are canonic), hence their $g^*$-images form a null-sequence
of $g$-cells from $\hg^*$, hence by Claim 1 the sets
$\vp(g^*(\vp^{-1}(L_i)))=f^*(L_i)$ form a null-sequence of $f$-cells from
$\hx^*$. Since $x_i\to x$ and $L_i$ form a null-sequence, we can find points
$z_i\in \ol{L_i}\cap X$ with $z_i\to x$. Then since $f$  does not change on
$X$, $f^*(z_i)=f(z_i)\to f(x)$. On the other hand $d(f^*(x_i), f^*(z_i))\to 0$
because $\{f^*(L_i)\}$ is a null-sequence. Hence $f^*(x_i)\to f^*(x)=f(x)$ as
desired.

Suppose now that $f^*(x_i)\not \to f^*(x)=f(x)$. We may assume
(by the previous paragraph), that $\dia L_i\ge \e''$ for some $\e''>0$.
Then since $\hx$ and the corresponding foliation $\hg$ are canonic, we may
assume that $L_i\to L\in \hx$ and $M_i\to M\in \hg$ with both $L, M$ being
non-degenerate. Clearly, points $\vp^{-1}(x_i)$ converge to a point
$\vp^{-1}(x)\in \ol{M}\cap S^1$ which implies that there are points $z_i\in
\ol{L_i}\cap X$ such that $z_i\to x$ and $\vp^{-1}(z_i)\to \vp^{-1}(x)$.
Therefore, by the construction, $d(g^*(x_i), g^*(z_i))\to 0$. By
Lemma~\ref{continu} this implies that $d(\vp(g^*(x_i)),
\vp(g^*(z_i)))=d(f^*(x_i), f^*(z_i))\to 0$. Since $f^*(z_i)=f(z_i)\to
f^*(x)=f(x)$, we finally conclude that $f^*(x_i)\to f^*(x)$ as desired. Thus,
the map $f^*$ has all the required properties and the theorem is proven.
\end{proof}

\section{Converging arcs and fixed points}\label{main-section}

By (A1) - (A7), $X$ is an indecomposable continuum containing no subcontinua
$Y$ with $f(Y)\subset T(Y)$. In particular, $X$ contains no invariant
subcontinua not equal to $X$. By the construction, if we prove our Main Theorem
for $f^*$, it will hold for $f$ too. Thus, in what follows we denote the map
$f^*$ constructed in Section~\ref{advaprel} by simply $f$. Also, set $R_\be=R$.
We deal a lot with subsegments of $R$ and from now on skip the subscript $\be$
in denoting them (so that $[a, b]$ means in fact $[a, b]_\be$ etc).
Similarly we denote the order $<_\be$ in $R$ simply by $<$
 (the situation considered in this section allows
us to do so without causing any ambiguity). Sometimes however we need to deal
with subarcs of other arcs/rays/lines, not contained in $R$. In that case we
indicate this with a subscript; thus, if $T$ is an arc/ray/line and $u, v\in T$
then by $[u, v]_T$ we mean the closed subarc of $T$ with endpoints $u, v$ (for
rays $R_\al$ we use the usual notation $[a, b]_\al$). Denoting subsets of
$R$ we use $\iy$ in the obvious sense (thus, $(x, \iy)$ is the subray of $R$
consisting of all points $y\in R$ with $y> x$).

By Section~\ref{advaprel}, we may assume that for some $z\in R$ the tail $(0,
z]$ of $R$ is invariant in the sense that $f|_{(0, z]}:(0,z]\to (0,f(z)]$ is an
embedding so that for all $x\in(0,z]$, $f(z)>z$. The ray $R$ is ordered from
infinity towards $X$; if $u, v\in R$ and $u<v$, say that $v$ is
\emph{$R$-closer to $X$ ($u$ is closer to $\iy$) than $u$}. We also say
\emph{$R$-closer} speaking of points on $R$ and meaning the order on $R$.


We need the following lemma.

\begin{lemma}\label{nd}
If $Z\subset X$ is nowhere dense in $X$ then for any $n$ the set $Z\cup
f(Z)\cup \dots f^n(Z)$ is nowhere dense in $X$ too. Hence, since $X\cap
\tau(X)$ is a closed and nowhere dense subset of $X$, for any
$n\in\mathbb{Z}^+$ there exists an open set $U\subset X$ such that for each
$i=0,1,\dots,n$, $f^i(U)\cap \tau(X)=\0$.
\end{lemma}

\begin{proof} Given a closed ball $B$ not containing the critical image,
we have $f^{-1}(B)=B'\cup B''$ with both $B', B''$ homeomorphic to $B$. Since
$Z\cap B'$ is nowhere dense in $B'\cap X$, then $f(Z\cap B')$ is nowhere dense
in $f(B'\cap X)\subset B\cap X$. Similarly, $f(Z\cap B'')$ is nowhere dense in
$B\cap X$ too. Hence $f(Z)\cap B$ which is the union of two nowhere dense in
$B\cap X$ sets is nowhere dense in $B\cap X$. This implies that $f(Z)$ is
nowhere dense in $X$ and proves, inductively, the first claim of the lemma.

Set $Z=X\cap \tau(X)$. Then the complement to $Z\cup f(Z)\cup \dots f^n(Z)$ is
a dense open subset $W$ of $X$ (by the first paragraph). On the other hand, $W$
consists of the points $x$ such that sets $x, f^{-1}(x), \dots, f^{-n}(x)$ are
disjoint from $Z$. Hence any point $y\in f^{-n}(x)$ is such that $y,$ $f(y),$
$\dots,$ $f^n(y)=x$ do not belong to $Z$. Hence if we take a small neighborhood $U$
of $y$ it will satisfy the requirements of the lemma.
\end{proof}

Now we are ready to prove our main theorem.







\begin{theorem}\label{main}
Suppose that $f:\C\to\C$ is a branched covering map such that the  absolute
value of the degree is at most $2$ and let $Y$ be a continuum such that $f(Y)\subset T(Y)$.
Then one of the following  holds.

\begin{enumerate}

\item The map $f$ has a fixed point in $T(Y)$.

\item The continuum $Y$ contains a \emph{fully invariant indecomposable}
continuum $X$ such that $X$ contains no subcontinuum $Z$ with $f(Z)\subset Z$;
moreover, in this case $\degree(f)=-2$.

\end{enumerate}

\end{theorem}

\begin{proof}
As before, we assume that (1) does not hold while the standard assumptions
(A1)-(A7) apply to an indecomposable continuum $X\subset Y$. We also
assume that the map has already been modified according to Theorem~\ref{hom}
and that therefore there exists a ray $R=R_\be$ with all properties of $R_\be$
as well as properties listed in Theorem~\ref{hom}. We may assume that $X$ is a
non-degenerate continuum containing no invariant subcontinua and such that
$f|_{T(X)}$ is fixed point free. Note that by (A7) $f^{-1}(X)=X\cup
\tau(X)\supsetneqq X$ is a continuum.

By Lemma~\ref{emptyint}, the set $X\cap \tau(X)$ is nowhere dense in $X$ while
$X\setminus \tau(X)=Q$ is a dense open subset of $X$. By Lemma~\ref{nd} we can
choose a point $p\in Q$ so that $f(p)=q\in Q, f(q)\in Q, f^2(q)\in Q$. We may
assume that $p$ is not equal to $c$ and its first preimages. Thus we can choose
a small neighborhood $V$ of $p$ such that $f^3|_V$ is a homeomorphism. Set
$U=f(V), U'=f(U), U''=f^2(U)$; we may assume that $X\cap(V\cup U\cup U'\cup
U'')\subset X\setminus \tau(X)$. Since the principal set of $R$ is $X$, the
sibling ray $\tau(R)$ is dense in $\tau(X)$ in the sense that $\ol{\tau(R)}\sm
\tau(R)=\tau(X)$. This implies that for some $\da>0$ the sibling ray $\tau(R)$ does
not come closer than $\da$ to $p, q, f(q)$ or $f^2(q)$. Hence we may assume
that $V$, $U$, $U'$ and $U''$ are all disjoint from $\tau(R)$.

Choose an $R$-defining family of crosscuts $C_t$ (see Section~\ref{baseprel}).
Since $R$ converges to $X$, there is a point $r\in V\cap R$ with $\ol{C_r}\subset
U$. Choose $r\in R$ to satisfy a few conditions. First, we may assume that $(r,
\iy)$ is a vertical line and $C_r$ is a horizontal segment. By
Theorem~\ref{hom}, we may assume that $f^4(r)\in R$ and
$f|_{\sh(C_{f^3(r)})}:\sh(C_{f^3(r)})\to \sh(f(\sh(C_{f^3(r)})$ is a
homeomorphism so that $f$ maps points on $R\cap \sh(C_{f^{3}(r)})$ to points on
$R$ closer to $\iy$. Let $W=\sh(C_{f^3(r)})$. We may also assume that $W$
contains no fixed points of $f$.

In the forthcoming arguments we move along the ray $R$ \emph{towards} $X$ and
use the terms like ``after'', ``before'' etc in the appropriate sense.
Figure~\ref{fig1} may help the reader to visualize the following construction.
Extend
$C_r$ a bit to the left while removing the part located to the right of $r$ to
create an arc $G\subset V$ disjoint from $(r, \iy)$. We may assume that $(0,
r)$ intersects $G$ infinitely often. To see this, take a sequence of points
on $R$ converging to the endpoint of $G$ distinct from $r$, draw an arc
through them all which ends at the left endpoint of $G$ and is disjoint from
$C_r$, and then add this arc to $G$. The added arc can be chosen arbitrarily
small. Then shorten $G$ a little by choosing its endpoint $t$ distinct from $r$
as the \emph{first point after $r$} (closest to $r$ in the sense of the order on $R$)
of $(0, r)$ intersecting the just extended $G$. Then $[r,t]\cup G$ is a
Jordan curve, and we may assume that $X$ meets both the unbounded component and
the bounded component of $\Complex\setminus\{[r,t]\cup G\}$. Indeed, we can
choose a point $y\in R$ and an essential crosscut $C_y$ such that $\ol{C_y}\cap
\ol{C_r}=\0$. Then we can construct $G$ so that $G\sm C_r$ is very small and
hence is disjoint from $C_y$. This implies that one of the two endpoints of
$C_y$ is inside the unbounded component and the other is in the bounded component of
$\Complex\setminus\{[r,t]\cup G\}$.

By Section~\ref{advaprel} the image of $G$ is the arc $f(G)=H\subset U$ which
grows out of $R$ at $f(r)=a$ and sticks out of $R$ to the right (the other
endpoint of $H$ is $f(t)=x$). By Section~\ref{advaprel} the image of $H$ is the
arc $f(H)=H'\subset U'$ which grows out of $R$ at $f(a)=a'$ and sticks out of
$R$ to the left (the other endpoint of $H'$ is $f(x)=x'$). Finally, we consider
the arc $f^2(H)=H''\subset U''$ which grows out of $R$ at $f^2(a)=a''$, sticks
out of $R$ to the right and has the other endpoint $f^2(x)=x''$. By the choice
of $U$ the segments $H, H'$ and $H''$ are pairwise disjoint. Moreover, by the
choice of $q$ the points of $X\cap (G\cup H\cup H'\cup H'')$ have no siblings
in $X$. To simplify the language we may assume that $G, H, H', H''$ are all
horizontal segments.

By the choice of $G$ the set $G\cap R$ is infinite. Hence the sets $H\cap R,
H'\cap R$ and $H''\cap R$ are infinite.
Also, recall that by the above made choices $\tau(R)$ is disjoint from $V, U,
U', U''$ and hence from $H$ and $H'$. Clearly, there are irreducible
sub-segments of $R$ connecting points of $H$ and $H'$. We call such segments of
$R$ \emph{prime segments}, or simply \emph{primes}. Two distinct primes
intersect at most at one of their endpoints (in this case one can call them
\emph{concatenated}).

The idea of what follows is to consider primes and their pullbacks. We show
that there is a ``monotone'' sequence of primes $P_0, P_{-1}, P_{-2}, \dots$ in
the sense that their endpoints on $H$ and $H'$ are ordered monotonically.
Moreover, the primes have the property that $f(P_{-n-1})\subset P_{-n}$. Then
these primes converge to a limit continuum $K$ with $f(K)\subset K$.
However, the monotonicity implies that $P_0$ cuts the limit continuum $K$ off
some points of $X$ and hence $K\ne X$, a contradiction with the assumption that
$X$ contains no proper invariant subcontinuum. Defining the desired sequence of
primes requires some purely geometric considerations in the plane.

Our arguments are based upon the observation that if moving along $R$ from
infinity towards $X$ we meet a point then before that we must have met the
image of this point because the points in $R$ map towards infinity (i.e.,
$f(z)>z$ for $z\in R\cap W$). By the construction $R$ passes through $x$. Let
us show that this is the first time $R$ intersects $H$ after $r$. Indeed,
otherwise there is a point $z>x, z\in H\cap R$. The point $z$ has the preimage
$z'\in G$ which cannot belong to $\tau(R)$ because $\tau(R)$ is disjoint from
$U$. Hence $z'\in R$. Moreover, $z'>t$ because $z>x$. This contradicts the
choice of $t$. Similarly, $x'$ and $x''$ are the first times after $r$ when the
ray $R$ intersects arcs $H'$ and $H''$ respectively. This in turn implies, by
the same argument, that in fact $x'$ is the first point at which $R$ hits
$f^{-1}(H'')$.

\begin{figure}
\scalebox{.5}{
\includegraphics{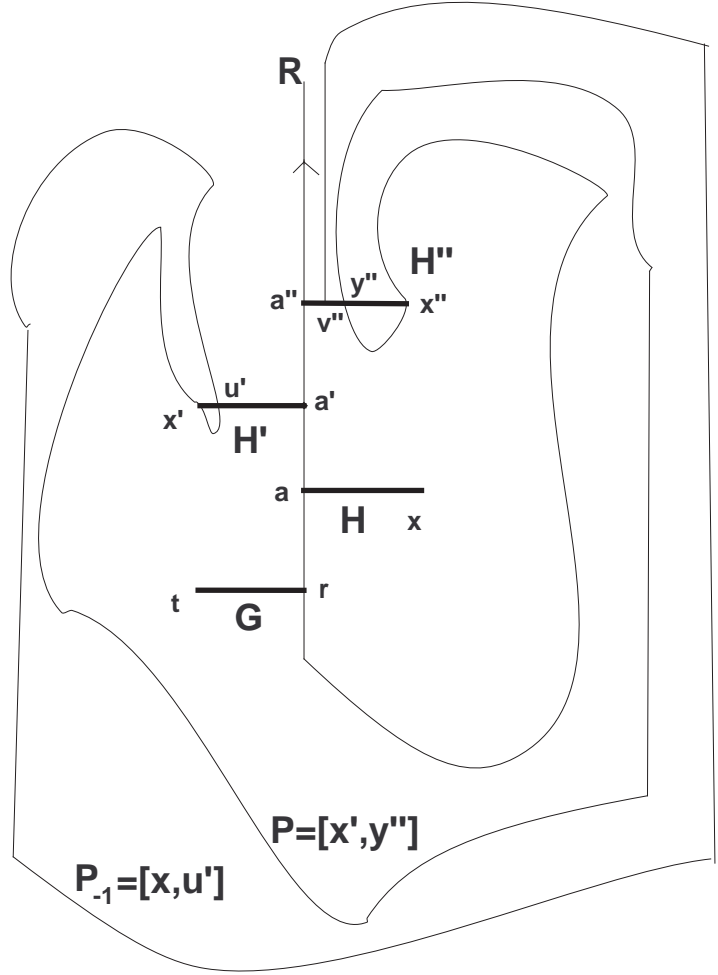}}
\caption{The ray $R=R_\beta$} \label{fig1}
\end{figure}

Clearly, the arc $[a'', x'']\subset R$ together with $H''$ forms a Jordan curve
$S$ which encloses an open Jordan disk $D$. Moreover, since both $H$ and $H''$
are located to the right of $R$ and since $H$ is disjoint from $S$ except for
$a$ we see that $(H\setminus \{a\})\subset D$. Now, the ray $R$ may have other
points of intersection with $H''$ after $x''$ and before it hits $H'$ for the
first time at $x'$. Denote by $y''$ the last point of $H''$ before $x'$.
Clearly, the ray points up as it finally exits $D$ at $y''$ for otherwise the
point $y''$ would not be the last point on $R$ before it hits $H'$ (we use the
fact that $R$ can only exit $D$ through $H''$ and that $H'$ is outside $D$).


Let $P$ be the subarc of $R$ given by $[y'', x']$. It follows that $[a'',
y'']_{H''}\cup P\cup H' \cup [a', a'']$ is a simple closed curve $T$ which
encloses a disk $\hD$ (recall that intervals without subscripts are subarcs of $R$).
 Observe that the
simple closed curves $S=[a'', x'']\cup H''$ and $T$ have an arc
$[a', a'']\cup [a'', y'']_{H''}$ in common and therefore form a
\emph{$\ta$-curve}. Observe also, that by the construction $D\subset \hD$ (see
Figure~\ref{fig1}).

Repeating the above arguments with obvious changes in notation we see that
after $x'$ the ray $R$ may have intersections with $H'$, then it finally goes
off $H'$ at a point $u'\in R\cap H'$ and then it hits $H$ for the first time at
the point $x$. We claim, that as $R$ goes off $H'$ at $u'$, it points up and moves
outside $\hD$. Indeed, suppose otherwise. Then the ray $R$ after $u'$ goes
inside $\hD$ and crosses $H$ at $x$ for the first time while not crossing $H'$
before that anymore. Consider the Jordan disk $D'$ whose boundary is formed by
$H'$ and the subsegment $[a', x']$ of $R$. Since $R$ cannot intersect itself
and, by the assumption, it does not intersect $H'$ after $u'$ anymore before it
intersects $H$, we see that $[u', x]\subset D'$. However by the construction
the point $x$ lies outside $D'$, a contradiction.

Thus, \emph{after $u'$ the ray $R$ goes up}. By the choice of $u'$ as the last
point on $H'$ on $R$ before $R$ hits $H$, it follows that $R$ has to penetrate
$D\ni x$ through $H''$ in order to reach out to $x$. Since $P=[y'', x']$
shields the subarc $[y'', x'']_{H''}$ of $H''$ from $R$ then the first point of
intersection between $R$ and $H''$ after $u'$ has to be a point $v''\in [a'',
y'']_{H''}$. Clearly, $R$ approaches $H''$ from above before it hits $H''$ at
$v''$. As we continue towards $H$, the ray $R$ after $v''$ may have more
intersections with $H''$, but then it finally hits $H$ at $x$. This creates our
first prime $P_0=[u', x]$ on which we have a subarc $[u', v'']$ with $v''\in
[a'', y'']_{H''}$. The prime $P_0$, together with the arc $[u',
a']_{H'}\cup [a', a]\cup H$, forms a simple closed curve $Y$ which encloses a
disk $L_0$.

Let us now define the first pullback of $P_0$. It follows that the ``zigzag''
arc $I=H'\cup [a', a]\cup H$ can be pulled back to the arc $J=G\cup [a, r]\cup
H$ (this pullback is simply a restriction of the corresponding branch of the
inverse function which is a homeomorphism). Observe that all points of $R\cap
I$ then pullback to points of $R\cap J$ (otherwise there will have to be points
of $\tau(R)$ in $J$ which is impossible). In particular, there exists a point
$u\in H\cap R$ with $f(u)=u'$. This pullback can be then extended onto $P_0$,
say, starting at $x$ and then by continuity. Let us show that this results into
a subarc of $R$ which connects $t\in G$ to $u\in H$. Indeed, under this
pullback the point $x$ pulls back to $t$. Recall, that by the construction
$R$ hits $G$ at $t$ for the first time.
Hence $P_0$ pulls back to an arc $Q$ which
at least around $t$ is a subarc of $R$, and hence overall (as a set) is a
subarc of $R$ too. The other endpoint of this subarc of $R$ should be the
unique preimage of $u'$ belonging to $R$, and by the shown above this can only
be the point $u$. Moreover, since at $u'$ the prime $P_0$ points up, so does the
arc $Q$ at $u$.

Observe also that $Q$ cannot intersect $H$ at more than one point since
otherwise its image $P_0$ will intersect $H'$ at more than one point. Therefore
the arc $Q$ exits $L_0$ at $u$ only to penetrate back into $L_0$ later through
$H'$ in order to reach out to $t\in G$. Denote by $s$ the closest to $u$ on $Q$
point of $H'$ and show that $s=v'$ is the unique preimage of $v''$ on $H'$.
Indeed, since $\tau(R)$ is ``far away'' from $V, U, U'$ and $U''$, then
$\tau(G\cup H\cup H'\cup H'')$ is disjoint from $R$. In particular, there are
no points of $\tau(H')$ in $[u, s)\subset Q$. On the other hand, there are no
points of $H'$ in $[u, s)$ by the choice of $s$. Therefore there are no points
of $f^{-1}(H''$ in $[u, s)$ which implies that there are no points of $H''$ in
$[u', f(s))$ while $f(s)\in H''$. By the definition of $v''$ this implies that
$f(s)=v''$ and hence $s=v'$ as desired. Moreover, $u\in H$ is closer to $a$
\emph{on $H$} than $x$ and $v'$ is closer to $a'$ on $H'$ than $u'$. Indeed,
the former is obvious. Also, as we pointed out before, $Q$ exits $L_0$ at $u$
and then it can only come back into $L_0$ through $[u', a']_{H'}$ so that
indeed $v'$ is closer to $a'$ on $H'$ than $u'$.

The arc $[u, v']\subset Q$ is then declared to be the next prime $P_{-1}$. By
the construction, its \emph{image} is a subarc of $P_0$. Moreover, $P_{-1}$ connects
$H$ and $H'$ in a specific way, namely so that the initial small segments at
the endpoints of $P_{-1}$ point up compared to the horizontal arcs $H$ and $H'$
respectively. To make the notation consistent let us from now on denote the
endpoints of $P_0$ by $\al_0=x, \be_0=u'$ and the endpoints of $P_{-1}$ by
$\al_1=u$ and $\be_1=v'$.
Observe that
$G\subset L_0$. The endpoints of $P_{-1}$ are located so that
$\al_1$ is closer to $a$ on $H$ than $\al_0$ and $\be_1$ is closer to $a'$ on
$H'$ than $\be_0$.

The above established  properties of primes can be used in the inductive process
showing that we can construct a sequence of primes with  similar
properties. Namely, suppose that we already have a finite sequence of pairwise
disjoint primes $P_0, P_1, \dots, P_{-n}$ such that the following holds.

\begin{enumerate}

\item $P_{-i}=[\be_i, \al_i]$ with $\be_i\in H', \al_i\in H$ and
$P_{-i}\cap (H'\cup H)=\{\al_i, \be_i\}$;

\item for each $i, 0\le i\le n-1$ the point $\be_{i+1}$ is closer to $a'$ than the point
$\be_i$ on the arc $H'$;

\item for each $i, 0\le i\le n-1$ the point $\al_{i+1}$ is closer to $a$ than the point
$\al_i$ on the arc $H$;


\item for each $i, 0\le i\le n-1$ we have that
$f(\al_{i+1})=\be_i$;

\item the initial segments of $P_{-i}$ at the endpoints of $P_{-i}$ point up;

\item for each $i, 0\le i\le n-1$ we have $f(P_{-(i+1)})\subset P_{-i}$.

\end{enumerate}

Let us show that then we can construct the next prime $P_{-n-1}$ so that all
these properties are satisfied. First though we locate a few points using the
fact that $f^3|_V$ is a homeomorphism. Since $f(\al_n)=\be_{n-1}$, we see that
$f([a, \al_n]_{H})=[a', \be_{n-1}]_{H'}$. Hence there is a preimage of
$\be_n\in H'$ in $H$, between $a$ and $\al_n$. Denote this preimage
$\al_{n+1}$. Also, choose $\zeta_{n+1}$ on $G$ so that $f(\zeta_{n+1})=\al_n$.
Finally, set $[\al_n, a]_H\cup [a, a']\cup [a', \be_n]_{H'}=Q_n$. Then it follows
from the location of the primes that $P_{-n}\cup Q_n=E_n$ is a Jordan curve
which encloses a Jordan disk $L_n$, and $L_0\subset L_1\dots \subset L_n$.
Moreover, $G\subset L_0$.

The point $\be_n$ has two preimages, $\al_{n+1}$ and $\tau(\al_{n+1})$. One of
them belongs to $R$, the other one belongs to $\tau(R)$. Since $\tau(R)$ is
disjoint from $V, U, U', U''$ then $\al_{n+1}\in R$. Similarly we see that
$\zeta_{n+1}\in R$. Hence the pullback $S_n$ of $P_{-n}$ within $R$ (we can talk
about it because by Theorem~\ref{hom} we assume that a tail of $R$ is
invariant) connects $\zeta_{n+1}$ and $\al_{n+1}$. Moreover, $S_n$ points up at the
points $\zeta_{n+1}$ and $\al_{n+1}$ because so does $P_n$ at their images, i.e. at
the points $\al_n$ and $\be_n$.

It follows that at $\al_{n+1}$ the arc $S_n$ exits $L_n$ and that $S_n$
intersects $H$ only at $\al_{n+1}$ (otherwise $P_n$ would intersect $H'$ at
more than one point $\be_n$). Since the other endpoint of $S_n$ is
$\zeta_{n+1}\in L_n$, it must enter back into $L_n$, and by the above it can
only do so through $H'$ closer to $a'$ than $\be_n$ (the rest of $H'$ is
shielded from $S_n$ by $P_{-n}$). Follow $S_n$ from $\al_{n+1}$ on towards
$\zeta_{n+1}$ until it meets $H'$ for the first time. Denote the closest to
$\al_{n+1}$ \emph{on $R$} point of $S_n$ which belongs to $H'$ by $\be_{n+1}$.
Then the arc $[\al_{n+1}, \be_{n+1}]=P_{-n-1}$ satisfies all the conditions on
primes listed above. Thus, we were able to make the step of induction which
proves the existence of an \emph{infinite} sequence of primes
$\{P_{-i}\}_{i=0}^\iy$ with the above listed properties.

By the construction the sequence of primes $\{P_{-i}\}$ converges to a
continuum which we denote $Z$. Indeed, the endpoints of primes $\al_n, \be_n$
converge to points $\al\in H\cap X, \be\in H'\cap X$ respectively. Choose $Z$
as the limit of a subsequence of primes, then choose a small neighborhood $M$
of $Z$, and then choose $P_{-N}$ so that the arc $[\al, \al_N]_H\cup P_{-N}\cup
[\be_n, \be]_{H'}\subset M$. It follows that then the Hausdorff distance
between $P_{-k}$ and $Z$ for any $k>N$ must be small and implies that $Z$ is
the limit (in the Hausdorff metric) of the sequence of primes $P_{-n}$.

Obviously, $Z\subset X$. Moreover, by the construction there are points of $X$
inside $L_0$ while $Z$ is disjoint from $L_0$. Therefore $Z\ne X$. However, by
continuity the fact that $f(P_{-(i+1)})\subset P_{-i}$ for every $i$ implies
that $f(Z)\subset Z$ which contradicts the minimality of $X$. Hence we may
finally conclude that the assumption of $X$ not being fully invariant fails. In
other words, $X$ is fully invariant (i.e. $f^{-1}(X)=X=f(X)$) as desired.
\end{proof}

We would like to make a few concluding remarks here. The fact that $X$ is fully
invariant allows us to work with the \emph{entire} uniformization plane. Recall
that $\vp:\disk^\iy\to U_\iy(X)$ is a Riemann map with $\vp(\iy)=\infty$. Then
the map $f$ is transported to the uniformization plane on which we obtain a
well-defined map $g(x)=\vp^{-1}\circ f\circ \vp(x), x\in \idisk$. This
construction is exactly the same as the standard construction from complex
dynamics; it was used in a more complicated situation of a non-fully invariant
continuum in \cite{fokkmayeovertymc07} as well as above in Section 4 of this
paper (though in that case the map $g$ was not considered on the entire
$\idisk$).

By the results of \cite{fokkmayeovertymc07}, $f$ induces a covering map
$G:S^1\to S^1$ on the circle of prime ends of $T(X)$ (i.e., $g$ continuously
extends over $S^1=\bd \idisk$ as a covering map of the circle). It is easy to
check that $\deg(G)=-2$. Hence, $G$ has exactly three fixed points
$\{\al_1,\al_1,\al_3\}$ in $S^1$. Suppose that $C_n$ is a fundamental chain of
crosscuts of the prime end $\al_j$. Since $\diam(C_n)\to 0$ and $f$ is fixed
point free on $T(X)$, for all $n$ sufficiently large, $f(C_n)\cap C_n=\0$.
Hence from that point on either $f(C_n)$ separates $C_n$ from infinity in
$\C\sm T(X)$ (the points are ``repelled'' from $X$ in the sense of the order on
the ray $R_{\al_j}$ in which case we have the so-called \emph{outchannel}
defined more precisely in \cite{fokkmayeovertymc07}), or $C_n$ separates
$f(C_n)$ from infinity in $\C\sm T(X)$ (the points are ``attracted'' towards
$X$ in the sense of the order on the ray $R_{\al_j}$ in which case we have the
so-called \emph{inchannel} defined more precisely in
\cite{fokkmayeovertymc07}). By \cite{fokkmayeovertymc07} there exists exactly
one outchannel, therefore two of the fixed prime ends must correspond to
inchannels. Hence the induced map $G$ on the circle of prime ends has
degree $-2$, exactly
one repelling fixed point and two attracting fixed points.
This details the dynamics in the neighborhood of $X$.

\bibliographystyle{amsalpha}
\providecommand{\bysame}{\leavevmode\hbox to3em{\hrulefill}\thinspace}
\providecommand{\MR}{\relax\ifhmode\unskip\space\fi MR }
\providecommand{\MRhref}[2]{%
  \href{http://www.ams.org/mathscinet-getitem?mr=#1}{#2}
}
\providecommand{\href}[2]{#2}
\bibliography{c:/lex/references/refshort}

\begin{thebibliography}{FMOT07}


\bibitem[Aki99]{akis99}
V.~Akis, \emph{On the plane fixed point problem}, Topology Proc. \textbf{24}
(1999), 15--31.

\bibitem[BO06]{bo06} A. Blokh, L. Oversteegen,
\emph{Monotone images of Cremer Julia sets} (2006),
to appear in Houston Journal of Mathematics, arXiv:0809.1193.

\bibitem[Bel67]{bell67}
H.~Bell, \emph{On fixed point properties of plane continua}, Trans. A.~M.~S.
\textbf{128} (1967), 539--548.

\bibitem[Bel78]{bell78}
\bysame, \emph{A fixed point theorem for plane homeomorphisms}, Fund. Math.
\textbf{100} (1978), 119--128, See also: Bull. A.~M.~S. 82(1976), 778-780.

\bibitem[Bon04]{boni04}
M.~Bonino, \emph{A {B}rouwer like theorem for orientation reversing
homeomorphisms of the sphere}, Fund. Math. \textbf{182} (2004), 1--40.

\bibitem[Bro12]{brou12a}
L.~E.~J. Brouwer, \emph{Beweis des ebenen {T}ranslationessatzes}, Math. Ann.
\textbf{72} (1912), 35--41.

\bibitem[Bro84]{brow84a}
Morton Brown, \emph{A new proof of {B}rouwer's lemma on translation arcs},
Houston J. of Math. \textbf{10} (1984), 35--41.

\bibitem[CL51]{cartlitt51}
M.~L. Cartwright and J.~E. Littlewood, \emph{Some fixed point theorems}, Annals
of Math. \textbf{54} (1951), 1--37.

\bibitem[CMT]{chmatuty06} D. Childers, J. Mayer, M. Tuncali and E. Tymchatyn,
\emph{Indecomposable continua and the Julia sets of rational maps}, Contemp.
Math. \textbf{396} (2006), 1--20.

\bibitem[Fat87]{fath87}
Albert Fathi, \emph{An orbit closing proof of {B}rouwer's lemma on translation
arcs}, L'enseignement Math\'{e}matique \textbf{33} (1987), 315--322.

\bibitem[FMOT07]{fokkmayeovertymc07}
R.~J. Fokkink, J.C. Mayer, L.~G. Oversteegen, and E.D. Tymchatyn,
\emph{The plane fixed point problem}, arXiv:0805.1184v2.

\bibitem[Fra92]{fran92}
J.~Franks, \emph{A new proof of the {B}rouwer plane translation theorem},
Ergodic Theory and Dynamical Systems \textbf{12} (1992), 217--226.

\bibitem[Gui94]{guil94}
L.~Guillou, \emph{Th\'eot\`eme de translation plane de {B}rouwer et
g\'en\'eralisations du th\'eot\`eme de {P}oincar\'e-{B}irkhoff}, Topology
\textbf{33} (1994), 331--351.

\bibitem[Ili70]{ilia70}
S.~D. Iliadis, \emph{Location of continua on a plane and fixed points}, Vestnik
Moskovskogo Univ. Matematika \textbf{25} (1970), no.~4, 66--70, Series {I}.

\bibitem[Kra74]{kras74}
J.~Krasinkiewicz, \emph{On internal composants of indecomposable plane
continua}, Fund. Math. \textbf{84} (1974), 255--263.

\bibitem[KP94]{kulkpink94}
R.~S. Kulkarni and U. Pinkall.
\newblock A canonical metric for {M}\"{o}bius structures and its applications.
\newblock {\em Math. Z.}, 216(1):89--129, 1994.

\bibitem[Mil00]{miln00}
J.~Milnor, \emph{Dynamics in one complex variable}, second ed., Vieweg,
Wiesbaden (2000).

\bibitem[OT08]{overtymc07}
Lex~G. Oversteegen and E.D. Tymchatyn.
\newblock Extending isotopies of planar continua,
\newblock arXiv:0811.0364v1.

\bibitem[Pom92]{pom92} Ch. Pommerenke, \emph{Boundary Behaviour of
Conformal Maps,} Springer-Verlag (1992).

\bibitem[Rog98]{roge98}
J.~T. Rogers, Jr., \emph{Diophantine conditions imply critical points on the
  boundaries of {S}iegel disks of polynomials}, Comm. Math Phys. \textbf{195}
  (1998), 175--193.

\bibitem[Sie68]{siek68}
K.~Sieklucki, \emph{On a class of plane acyclic continua with the fixed point
property}, Fund. Math. \textbf{63} (1968), 257--278.

\bibitem[Ste35]{ster35} Sternbach, Problem {\bf 107} (1935),
in: {\em The {S}cottish {B}ook: {M}athematics from the {S}cottish
{C}af\'e}, Birkhauser, Boston, 1981, 1935.

\bibitem[Thu85]{thu85} W. Thurston, \emph{The combinatorics of iterated rational maps}
(1985), in: ``Complex dynamics: families and friends'', ed. by D. Schleicher,
A K Peters (2008), Wellesley, MA, pp. 1-108.

\bibitem[Why42]{whyb42}
G.~T. Whyburn, \emph{Analytic topology}, vol.~28, AMS Coll. Publications,
Providence, RI, 1942.


\end{thebibliography}

\end{document}